\documentclass[11pt]{amsart} 
\usepackage[margin=1.25in]{geometry}

\usepackage{amssymb,upgreek}
\usepackage{amsmath}
\usepackage{url}
\usepackage{bm}
\usepackage{hyperref}
\usepackage{enumerate}
\usepackage{xcolor}

\usepackage{tikz}

\usepackage{color}
\usepackage{mathtools} 
\usepackage{caption}
\usepackage{subcaption}
\usepackage{fancyhdr}
\usepackage{mathrsfs}
\newtheorem{theorem}{Theorem}[section] 
\newtheorem{lemma}[theorem]{Lemma}

\newtheorem{corollary}[theorem]{Corollary}
\newtheorem*{theorem*}{Theorem}
\newtheorem*{lemma*}{Lemma}
\newtheorem*{corollary*}{Corollary}
\newtheorem*{maintheorem}{Theorem \ref{thm:main}}
\newtheorem*{kawauchicor}{Corollary \ref{cor:kawauchidisks}}
\newtheorem*{twistcor}{Corollary \ref{cor:twistcor}}
\newtheorem*{schoenfliestheorem}{Theorem \ref{thm:schoenfliesball}}

\theoremstyle{definition}
\newtheorem{definition}[theorem]{Definition}
\newtheorem{remark}[theorem]{Remark}

\newtheorem{question}[theorem]{Question}

\newtheorem*{remark*}{Remark}
\newtheorem*{definition*}{Definition}
\newtheorem*{remarks*}{Remarks}
\newtheorem*{addenda*}{Addenda}

\newcommand{\bit}[1]{\textbf{\textit{#1}}} 

\newcommand{\pt}{\textup{pt}}

\newcommand{\sto}{\!\!\xymatrix@C=1em{{}\ar@{~>}[r]&{}}\!\!}

\newcommand{\interior}{\textup{int}}

\newcommand{\cs}{\mathop\#}

\title{Doubles of Gluck twists: A five-dimensional approach}
\author{David Gabai, Patrick Naylor, and Hannah Schwartz}
\date{\today}

\begin{document}

\maketitle

\begin{abstract} 
Using a 5-dimensional perspective, we balance algebraic and geometric handle cancellation to show that doubles of Gluck twists of certain 2-spheres with two minima are standard. This includes all 2-spheres which are unions of ribbon discs, one of which has undisking number one. As an application, we produce new examples of Schoenflies balls not known to be standard.
\end{abstract}

\section{Introduction}

The operation of \emph{Gluck twisting} a smoothly embedded $2$-sphere $S\subset S^4$ (the only non-trivial notion of surgery along a $2$-sphere) has been studied extensively since Gluck's seminal work \cite{Glu62} in the 1960s. The resulting smooth $4$-manifold $\Sigma_S$ is called the \emph{Gluck twist of $S$}. Each Gluck twist is a homotopy $4$-sphere and homeomorphic to $S^4$ by \cite{Fre82sc}. Many families of knotted 2-spheres have the property that their Gluck twists are \emph{standard}, i.e. diffeomorphic to $S^4$, but it remains an open question whether all Gluck twists are standard. Hence, the study of this family of 4-manifolds is motivated mainly by the ongoing search for a counterexample to the smooth 4-dimensional Poincare conjecture.

To show that a Gluck twist is standard, one usually manipulates the 4-manifold $\Sigma_S$ directly using handle calculus and appealing to specific properties of the knotted sphere $S$. In this paper, we use a $5$-dimensional approach to show that certain Gluck twists of the form $\Sigma_{S \cs -S}$ are standard by proving that $\Sigma_S^\circ \times I$ is diffeomorphic to the standard $5$-ball. Our main theorem is the following.

\begin{maintheorem}
Suppose that a 2-sphere $S$ admits a decomposition into two ribbon disks, one of which has undisking number equal to one. Then, the Gluck twist $\Sigma_{S \cs -S}$ is standard.  
\end{maintheorem}

The low \emph{undisking number} (see Definition \ref{def:bandunknottingnumber}) requirement characterizes a subset of spheres admitting a Morse function with two maxima and any number of minima; it is readily satisfied by a plethora of spheres whose Gluck twists are not known to be standard. It also covers other historically interesting cases. For instance, in Remark \ref{rem:gompfex}, we use Theorem \ref{thm:main} to obtain an alternate proof that the double of the Cappell-Shaneson homotopy $4$-sphere $\Sigma_0$ is standard. We also show that most ribbon disks in Kawauchi's table \cite[F.5]{Kaw96book} have undisking number $1$, and so obtain the following corollary.
 
\begin{kawauchicor}
Suppose that a 2-sphere $S$ has one ribbon hemisphere given by a ribbon disk for one of the following knots appearing in \cite[F.5]{Kaw96book}. 
\[\{6_1,8_8,8_9,8_{20}, 9_{27}, 9_{41}, 9_{46}, 10_{48},10_{75}, 10_{87},10_{129},10_{137}, 10_{140},10_{153},10_{155}\}\]
Then, the Gluck twist $\Sigma_{S \cs -S}$ is standard. 
\end{kawauchicor}

We also consider the family of ribbon disks $D_K$ for the connected sums $K\#-K$ obtained by spinning $K$ through the 4-ball. Applying Theorem \ref{thm:main} gives the following corollary.

\begin{twistcor}
Suppose that $K$ is a twist knot. If a 2-sphere $S$ has one ribbon hemisphere given by $D_K$, then the Gluck twist $\Sigma_{S \cs -S}$ is standard. 
\end{twistcor}

If $X$ is a homotopy 4-sphere, then let $X^\circ$ denote the homotopy 4-ball which is the closed complement of a smoothly embedded 4-ball. A homotopy 4-ball is called a \emph{Gluck ball} if it arises from a Gluck twist, and a \emph{Schoenflies ball} if it embeds in the 4-sphere. A corollary to Theorem \ref{thm:main} is the following. 

\begin{schoenfliestheorem}
Suppose that a 2-sphere $S\subset S^4$ admits a decomposition into two ribbon disks, one of which has undisking number equal to one. Then, $\Sigma_S$ is a Schoenflies ball. 
\end{schoenfliestheorem}

Thus, we obtain many potentially interesting examples from the perspective of the Schoenflies problem, which we hope to explore in future work. 

\subsection*{Organization} The paper is organized as follows. In Section \ref{sec:definitions}, we define banded unlink diagrams for surfaces embedded in $S^4$ or $B^4$, and use these diagrams to construct specific handle decompositions of Gluck twists. Section \ref{sec:5d} describes how 5-dimensional 2- and 3-handles can be canceled under favorable circumstances. We state and prove Theorem \ref{thm:main} in Section \ref{sec:mainresults}. Finally, we conclude with some examples and questions in Section \ref{sec:examples}.

\subsection*{Acknowledgements} 
The first author was partially supported by NSF grant DMS-2003\allowbreak892, the second by an NSERC postdoctoral fellowship and Princeton University, and the third by NSF grant DMS-1502525 in addition to the NSF grant DMS-2204367.

\section{Gluck twists and banded unlink diagrams}\label{sec:definitions}

\subsection{Gluck twists of knotted 2-spheres} 

Throughout this paper, we work in the smooth orientable category. If $M$ is a closed manifold, we will write $M^\circ$ for the compact manifold obtained by removing a small open ball from $M$. We begin with the precise definition of a Gluck twist, first introduced by Gluck \cite{Glu62}.

\begin{definition}\label{def:glucktwist} Let $S$ be a $2$-sphere smoothly embedded in a $4$-manifold $X$ with trivial normal bundle $N\cong S^2\times D^2$. The \bit{Gluck twist of $X$ along $S$} is the closed 4-manifold $$\Sigma_S(X)= (X-\interior(N)) \cup_\tau N$$ where $\tau$ is the automorphism of $\partial N \cong S^2 \times S^1$ sending $(x, \theta) \mapsto (r_\theta(x), \theta)$, and $r_\theta$ is the rotation of $S^2$ by $\theta$ about some fixed axis. 
\end{definition}

By work of Gluck \cite{Glu62}, the map $\tau$ is essentially the unique non-trivial re-gluing map up to isotopy. Thus, although there appear to be choices in the above definition, $\Sigma_S(X)$ is well defined up to diffeomorphism. When $S$ is null-homotopic in $X$, the Gluck twist $\Sigma_S(X)$ is homotopy equivalent to $X$; hence when $X$ is simply connected, Freedman's work \cite{Fre82sc} implies that $\Sigma_S(X)$ is homeomorphic to $X$. For additional exposition, see \cite{KaPoRa22}. In the case $X=S^4$, we will simply write $\Sigma_S$ to denote the Gluck twist of $S^4$ along $S$. 

Since the introduction of this construction, it has remained an open question whether all Gluck twists on spheres in $S^4$ are \emph{standard}, i.e. diffeomorphic to the standard $4$-sphere. Families of $2$-spheres embedded in $S^4$ (which we also refer to as \emph{$2$-knots} or \emph{knots}) whose Gluck twists are known to be standard include:

\begin{itemize}
    \item spun knots, and more generally ribbon knots \cite{Glu62}
    \item twist spun knots \cite{Gor76} \cite{Pao78};
    \item knots 0-cobordant to knots with trivial Gluck twists \cite{MelvinThesis};
    \item even degree satellites of spheres with trivial Gluck twists, or arbitrary satellites in which both the pattern and companion knots have trivial Gluck twists \cite{Seu20};
    \item 2-knots which admit a regular homotopy to the unknot consisting of a single finger move and Whitney move \cite{NaySch20};
    \item other examples, including certain 2-knots arising in the Cappell-Shaneson construction, and the family of 2-knots studied by Nash and Stipsicz \cite{NasSti12}.
\end{itemize}

Gluck twists on $S^4$ that are not known to be standard suggest potential counterexamples to the smooth $4$-dimensional Poincar\'e conjecture. While the Gluck twisting operation has been used to change the smooth structure of some \emph{non-orientable} manifolds \cite{Akb88}, we are unaware of any example in a simply-connected or orientable setting. 

\subsection{Banded unlink diagrams for surfaces in the 4-sphere}

In the next section, we will draw handlebody diagrams for Gluck twists. To begin, we briefly review the notion of a banded unlink diagram. 

\begin{definition}\label{def:bud} We call the product $I \times I$ of two unit intervals a \bit{band} once either the $\{0,1\} \times I$ or $I \times \{0,1\}$ portion of its boundary is fixed as its \bit{attaching region}. A \bit{banded unlink diagram} $\mathcal D$ consists of an unlink $\mathcal U = U_1 \cup \dots \cup U_n$ in the $3$-sphere, together with a collection of $k$ 
bands $\beta=\{\beta_1, \dots, \beta_k\}$ embedded disjointly into $S^3$ so that only the \bit{attaching region} of each intersects $\mathcal U$. The \bit{resolution} of $\mathcal U$ along the bands in $\beta$ is the link $L_\beta = L_1 \cup \dots \cup L_m$ in the $3$-sphere obtained from $\mathcal U$ by replacing the attaching region of each band by the complementary portion of its boundary. The diagram $\mathcal D$ is called \bit{closed} if the link $L_\beta$ is an unlink. 
\end{definition}

Banded unlink diagrams corresponding to surfaces in $S^4$ were first studied by Yoshikawa \cite{Yos94}. Banded unlink diagrams of closed surfaces in arbitrary smooth 4-manifolds (i.e., a handlebody diagram of a $4$-manifold together with a banded unlink describing an embedded surface) were studied extensively by Hughes, Kim, and Miller \cite{HugKimMil18}.

\begin{definition}
Two embedded bands are \bit{dual} if they consist of the same embedding of $I \times I$, but have complementary attaching regions. The \bit{dual} of a closed banded unlink diagram is the diagram obtained by swapping the roles of the unlinks $\mathcal U$ and $L_\beta$, and taking the \bit{dual bands} $\beta'_1, \dots, \beta'_k$ in place of the original ones. 
\end{definition}

Examples of banded unlink diagrams are given in Figures \ref{fig:CSbud} and \ref{fig:foxbud}.

\begin{figure}[htbp]
    \begin{minipage}{0.49\textwidth}
    \centering
    \includegraphics[width=0.7\textwidth]{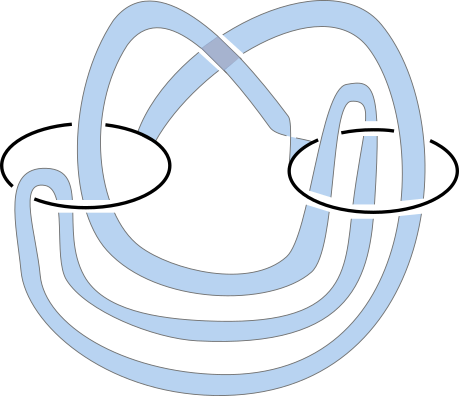}
    \end{minipage}
    \hfill
    \begin{minipage}{0.49\textwidth}
    \centering
    \includegraphics[width=0.7\textwidth]{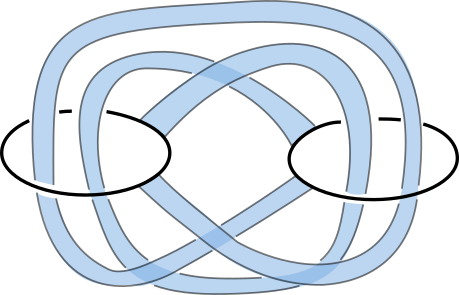}
    \end{minipage}

    \begin{minipage}[t]{0.49\textwidth}
    \caption{A banded unlink diagram corresponding to a ribbon disk in the $4$-ball for the knot $8_9$.}
    \label{fig:CSbud}
    \end{minipage}
    \hfill
    \begin{minipage}[t]{0.49\textwidth}
    \caption[]{A \emph{closed} banded unlink diagram corresponding to the 2-twist spun trefoil in $S^4$. This is Example 12 from \cite{Fox62book}.\protect\footnotemark}
    \label{fig:foxbud}
    \end{minipage}
\end{figure}

\footnotetext{This banded unlink diagram describes this 2-knot as the union of two different ribbon disks for the Stevedore knot, as in \cite{MelvinThesis}. That this is \emph{also} the 2-twist spun trefoil appears to have been first noticed by Litherland; another proof appears in \cite{Ka83}.}

A closed banded unlink diagram $\mathcal D$ determines an embedding of a closed surface $S_\mathcal D \subset S^4$ that is in Morse position with respect to the radial height function $h:S^4\to \mathbb{R}$ and such that $h\vert_{S_\mathcal{D}}$ is self-indexing. Viewing the saddles ``rectilinearly", the banded unlink diagram $\mathcal D$ is the pre-image of $S_\mathcal D$ in the level set $S^3 = h^{-1}(1)$ where the $1$-handles of the surface appear as flat bands in $\mathcal{D}$. The components of the unlink $\mathcal U$ are the ascending circles of the minima of $S_\mathcal D$, while (in the closed case) the components of the unlink $L_\beta$ are the descending circles of its maxima.
\footnote{Note that both the unlinks $\mathcal U$ and $L_\beta$ bound a unique set of spanning disks up to isotopy rel boundary in the $4$-ball, by \cite{Liv82}, and so the banded unlink completely specifies the embedding.} 
The dual to a closed banded unlink diagram determines the same embedding up to isotopy (but with the height function reversed). A banded unlink diagram in which $L_\beta$ is not the unlink determines a properly embedded ribbon surface $S_\mathcal D\subset B^4$ with boundary $L_\beta\subset S^3$. 

\begin{definition} \label{def:moves} 
We say that the embedded surface $S_\mathcal D \subset S^4$ \bit{corresponds to} the banded unlink diagram $\mathcal D$. The isotopy class of the corresponding surface is preserved by each of the following moves on banded unlink diagrams, illustrated in Figure \ref{fig:moves}: \bit{band slides}, \bit{band swims}, and \bit{cancellations/introductions}. 
\end{definition}

Yoshikawa \cite{Yos94} described these moves and conjectured they were sufficient to pass between any two diagrams corresponding to isotopic embedded surfaces in $S^4$. This question was answered affirmatively by Swenton \cite{Swe01}; Kearton and Kurlin also give an alternate proof \cite{KK08}. 

\begin{figure}[ht]
\includegraphics[width=\textwidth]{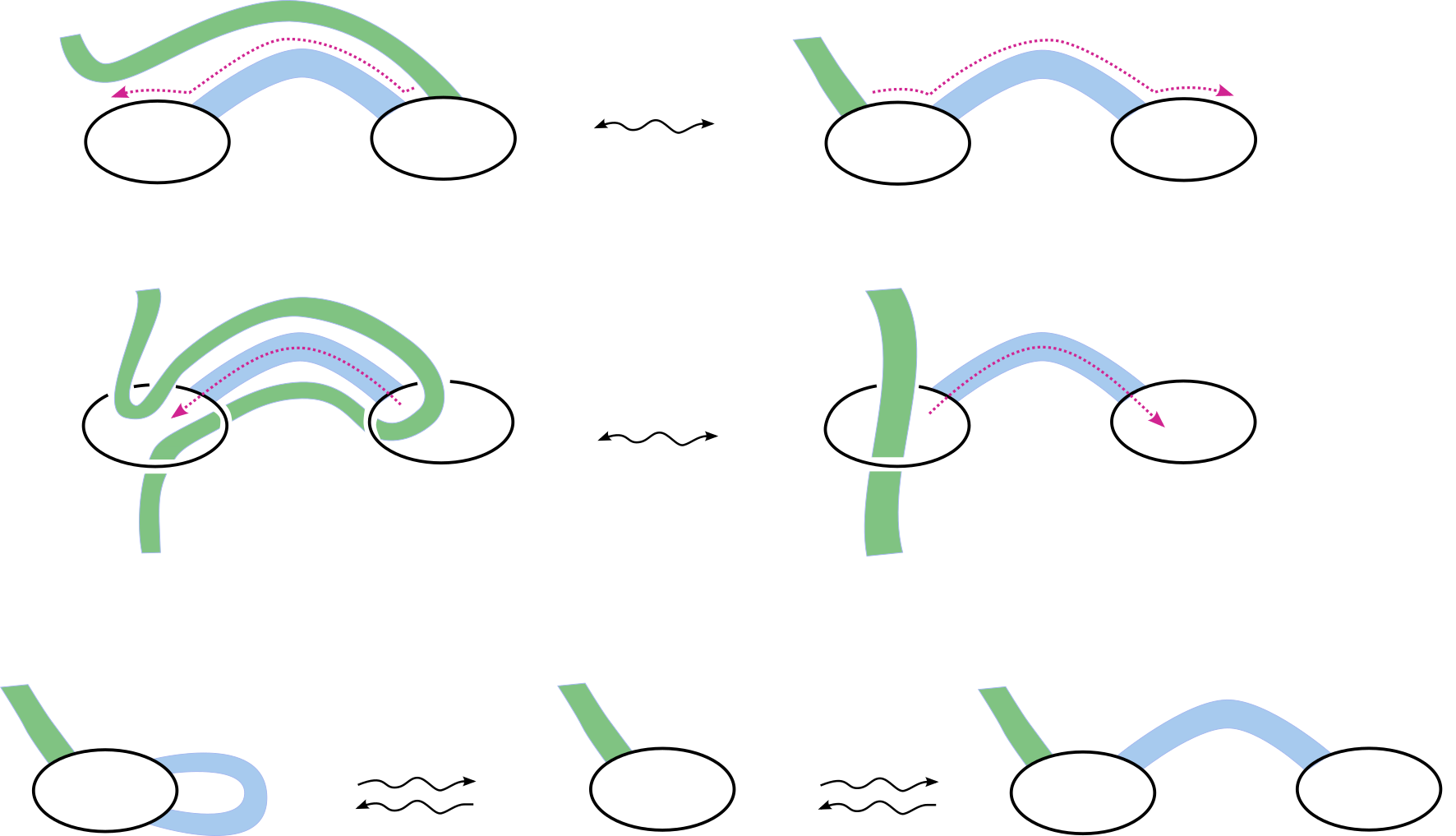}
\put(-262,103){\small Band swim}
\put(-262,196){\small Band slide}
\put(-262,196){\small Band slide}
\put(-330,-4){\small Intro 1/2}
\put(-334,22){\small Cancel 1/2}
\put(-193,22){\small Intro 0/1}
\put(-195,-4){\small Cancel 0/1}
\caption{Moves on a banded unlink diagram that preserve the isotopy class of the corresponding surface, as in Definition \ref{def:moves}.} 
\label{fig:moves}
\end{figure}
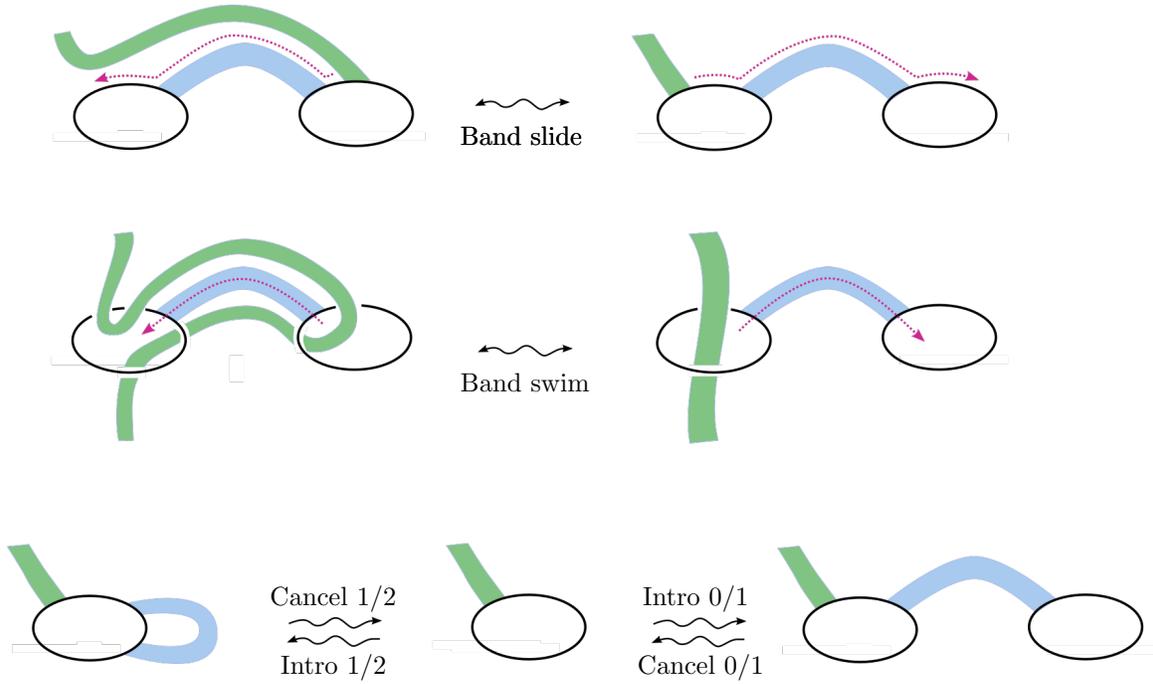


When the surface $S_{\mathcal D}$ corresponding to the banded unlink diagram $\mathcal D$ is a $2$-sphere, an Euler characteristic computation gives that $k-n +1 = m - 1$. In this case, it is a consequence of elementary Morse theory that the $k$ bands of the diagram admit a particularly convenient partition, giving the \emph{normal form} of \cite{KSS82}.
\begin{lemma}\label{fisfus}
 After a sequence of band slides (as in Definition \ref{def:moves}), the bands of a banded unlink diagram $\mathcal D$ corresponding to an embedded $2$-sphere can be relabelled and reindexed as a set of \bit{fusion bands} $\beta_1, \dots, \beta_{n-1}$ and a set of \bit{fission bands} $\delta_1, \dots, \delta_{m-1}$ such that 
 \begin{enumerate}
     \item[(1)] each fusion band $\beta_i$ connects $U_i$ to $U_{i+1}$;
     \item[(2)] the dual $\delta_i'$ of each fission band connects $L_i$ to $L_{i+1}$.
 \end{enumerate}
\end{lemma}

An important banded unlink diagram for our purposes is the simplest one. For simplicity (and for our purposes later in this paper) we will distinguish between isotopy of surfaces and isotopy of banded unlink diagrams; in the following definition, we do \emph{not} allow moves as in Figure \ref{fig:moves}.

\begin{definition}
\label{def:trivialbud}
A banded unlink diagram is called \bit{trivial} if it is isotopic to one of the form illustrated in Figure \ref{fig:trivialbud}. A trivial diagram corresponds to a slice disk for the unknot in the $3$-sphere, obtained by pushing its spanning disk into the $4$-ball. 
\end{definition}

\begin{figure}[ht]
\includegraphics[height=70pt]{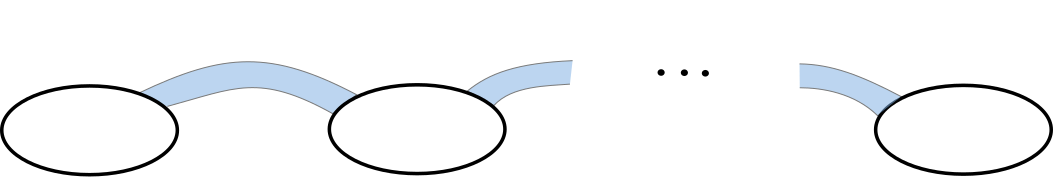}
\caption{A trivial banded unlink diagram, as in Definition \ref{def:trivialbud}. If the unlink $\mathcal U$ (black) has $n$ components, then there are exactly $n-1$ bands (blue), all of which are fusion bands. The link $L_\beta$ obtained by resolving the bands is simply the unknot.}
\label{fig:trivialbud}
\end{figure}

\subsection{Handle decompositions for Gluck twists}

In this section, we adapt the well-known ``rising water principle'' to obtain handle decompositions of the complement of a closed surface embedded in $S^4$ (presented as a banded unlink diagram). 

\begin{lemma}\label{def:stdhandledecomp} 
Let $\mathcal D$ be a closed banded unlink diagram for a closed connected surface $S_\mathcal{D}$. The complement $S^4-N(S_\mathcal{D})$ has a \bit{standard handle decomposition} corresponding to $\mathcal D$ built as follows:
\begin{enumerate}
    \item Attach $m$ $1$-handles $\mu_1, \dots, \mu_m$ to the $4$-ball for each component $L_i$ of $L_\beta$. Equivalently, remove regular neighbourhoods of a collection of disjointly embedded disks in $S^3$ bounded by the unlink $L_\beta$ (pushed into the $4$-ball). In ``dotted circle notation,'' this amounts to putting a dot on each component of the unlink $L_\beta$. 
    \item Attach $k$ $2$-handles $h_1, \dots, h_k$ with framing $0$ along unknotted curves ``dual" to each band $\beta_1,\dots \beta_k$, as in Figure \ref{fig:no3hs}.  
    \item Attach $3$-handles along $n$ disjointly embedded $2$-spheres $R_1, \dots, R_{n}$, corresponding to the disks bounded by the unlink $\mathcal{U}$, along with a 4-handle. Although 3-handles are not usually included in Kirby diagrams, their attaching regions can be drawn explicitly; see Remark \ref{rem:3handlespheres}. 
\end{enumerate}
\end{lemma}

\begin{proof} In the usual ``rising water'' method of drawing complements of surfaces in $S^4$ (see, e.g. \cite[\S 6.2]{GS99}), a 1-handle is added for each minimum, a 2-handle for each saddle point, and a 3-handle for each maximum. The algorithm above implements this process for the Morse function on $S_\mathcal{D}$ induced by the banded unlink diagram $\mathcal D$ flipped ``upside down" (i.e. the roles of the unlinks $\mathcal U$  and $L_\beta$ are reversed). This gives a clearer picture of the attaching spheres for the 3-handles (see Remark \ref{rem:3handlespheres}), at the expense of more complicated-looking 1-handles. \end{proof}

\begin{definition}\label{fisfushandles}
Let $\mathcal D$ be a banded unlink diagram for a knotted surface. The $2$-handles $h_1, \dots, h_{n-1}$ of the standard handle decomposition for either the complement of $\mathcal S_\mathcal D$ or its Gluck twist will be referred to as \bit{fusion $2$-handles}, as they are attached \emph{around} the fusion bands $\beta_1, \dots, \beta_{n-1}$ from Definition \ref{fisfus} (or equivalently, \emph{along} their duals). Likewise, the remaining $2$-handles $h_n, \dots, h_k$ will be referred to as \bit{fission $2$-handles} since they are attached around the fission bands.
\end{definition}

When the surface $S:=S_\mathcal{D} \subset S^4$ given by a banded unlink diagram $\mathcal{D}$ is a 2-sphere, the standard handle decomposition of the complement $S^4 - N(S)$ naturally provides one for the Gluck twist $\Sigma_S$. In this case, attaching 3-handles corresponding to \emph{all but one} of the maxima give a Kirby diagram for $S^4-\interior(N(S_\mathcal{D}))$. Attaching a $+1$-framed 2-handle along a meridian to any one of the 1-handles gives a diagram for the punctured Gluck twist $\Sigma_S^\circ$. See \cite[\S6.2]{AK16} for more details.

\begin{definition}\label{def:stdhandle} 
Let $\mathcal D$ be a closed banded unlink diagram for a $2$-sphere $S\subset S^4$. The \bit{standard handle decomposition of the Gluck twist $\Sigma_S$} corresponding to $\mathcal D$ is built from the standard handlebody decomposition of the complement $S^4 - N(S)$ by attaching an additional +1-framed 2-handle $h_*$ along the meridian of $S$ considered in the boundary $\partial(S^4 - \interior N(S))$ followed by a single $4$-handle. We will refer to this additional 2-handle $h_*$ as the \bit{Gluck 2-handle}. For convenience, we will assume that its attaching circle is the meridian to the first $1$-handle $\mu_1$.  
\end{definition}

Figure \ref{fig:no3hs} illustrates an explicit example of the process of producing a standard handle decomposition of a Gluck twist, starting from the banded unlink diagram in Figure \ref{fig:foxbud}. 

\begin{figure}[ht]
\includegraphics[height=150pt]{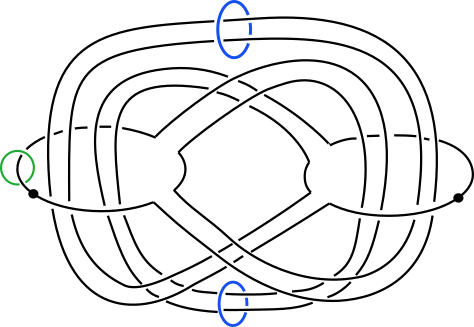}
\put(21,59){\small $\cup$ one $3$-handle}
\put(21,39){\small $\cup$ one $4$-handle}
\put(-233,70){\small $+1$}
\put(-110,154){\small $0$}
\put(-113,-10){\small $0$}
\caption{The ``standard handle decomposition" (as in Definition \ref{def:stdhandle}) for $\Sigma_{S_\mathcal{D}}$, where $\mathcal{D}$ is the diagram for the 2-twist spun trefoil in Figure \ref{fig:foxbud}. Alternatively, removing the +1 framed 2-handle and including an additional 3-handle gives a handle decomposition for $S^4-N(S_\mathcal{D})$.}
\label{fig:no3hs}
\end{figure}

\begin{figure}[ht]
\includegraphics[height=150pt]{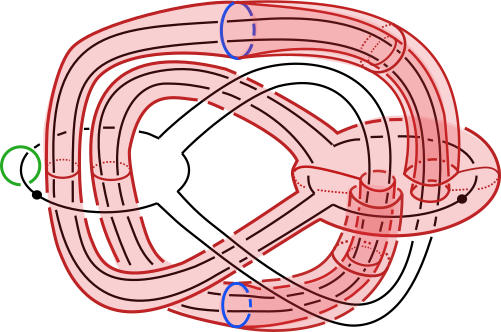}
\put(21,39){\small $\cup$ one $4$-handle}
\caption{The same standard handle decomposition in Figure \ref{fig:no3hs}, but with the attaching sphere for the 3-handle drawn in red. Note that it passes three times (once algebraically) over each of the blue 2-handles. As in Figure \ref{fig:no3hs}, all 2-handles are 0-framed except for the green ``Gluck 2-handle."}
\label{fig:yes3hs}
\end{figure}

\begin{remark}\label{rem:3handlespheres}
    Attaching spheres for 3-handles are usually not depicted in a handlebody diagram for a closed 4-manifold $X$, since they are attached uniquely up to diffeomorphism \cite{LauPoe72}. However, concrete attaching $2$-spheres can be determined in the following way. Begin by converting $\partial X^{(2)}$ (the boundary of the 0-, 1-, and 2-handles for $X$) to a standard picture for $\cs S^1 \times S^2$. Tracing the collection of spheres of the form $\{\pt\}\times S^2$ back to the handlebody diagram for $X^{(2)}$ gives the 3-handle attaching spheres $R_1,\dots,R_n\subset X^{(2)}$. 
    
    The attaching spheres can be explicitly constructed for these standard handle decompositions. Begin with the standard collection of spanning disks $D_1, \dots, D_n$ for the unlink $\mathcal U$. For each ribbon intersection of a band $\beta_i$ with the disk $D_j$, modify $D_j$ by replacing it with a new disk (labelled with the same name) obtained by deleting a neighbourhood of the ribbon intersection and replacing it with an annulus tubed along the band $\beta_i$ (choose a fixed direction to tube for each band) and capped off by a copy of the core of the $2$-handle $h_j$. This is illustrated in Figure \ref{fig:yes3hs}. For more details, the reader is encouraged to consult \cite[\S6.2]{GS99}.
\end{remark}

\subsection{Band passes and ribbon disks}

Unless otherwise indicated, we will only consider banded unlink diagrams up to isotopy of the bands rel boundary. In other words, we will consider such diagrams up to the \emph{ribbon presentations} they induce, rather than up to isotopy of the underlying surface. 

\begin{definition}
Let $\mathcal D$ be a banded unlink diagram. A \bit{band pass} of $\mathcal D$ is the local move illustrated in Figure \ref{fig:bandpass}, which produces a new banded unlink diagram related to $\mathcal D$ by passing a band through the unlink $\mathcal U$. 
\end{definition}

\begin{figure}[ht]
\includegraphics[height=100pt]{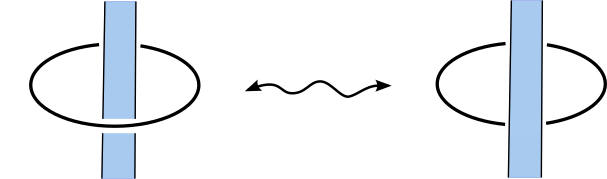}
\caption{A local picture of a band pass move on a banded unlink diagram, which passes a band through the unlink $\mathcal U$.}
\label{fig:bandpass}
\end{figure}

Band pass moves do \emph{not} preserve the isotopy class or even the boundary of the surface corresponding to a diagram, but they do preserve its genus. Consequently, we can define a new notion of ``complexity" for ribbon disks. 

\begin{definition}\label{def:bandunknottingnumber}
Let $D\subset B^4$ be a ribbon disk.  The \bit{undisking number $\delta(D)$} of $D$ is the least number of band pass moves required to convert a banded unlink diagram for $D$ to the trivial banded unlink diagram. 
\end{definition}

The undisking number of a ribbon disk $D$ gives an upper bound for its boundary's unknotting number $u(\partial D)$. In particular, we must have $2 u(\partial D) \leq \delta(D)$. Since band pass moves need not (and usually do not) take a closed banded unlink diagram to another closed banded unlink diagram, there is \emph{not} an analogous measure of complexity for closed surfaces. 

Many ribbon disks (especially well-known ones) for low-crossing knots have low undisking numbers. Of the ribbon disks given in \cite[F.5]{Kaw96book}, 15 of 21 have undisking number equal to one. The remaining ribbon disks (those bounded by the knots $10_3$, $10_{22}$, $10_{35}$, $10_{42}$, $10_{99}$, and $10_{123}$) likely have undisking number equal to two, a readily computed upper bound.  

\begin{figure}[htbp]
    \begin{minipage}{0.5\textwidth}
    \centering
    \includegraphics[width=0.7\textwidth]{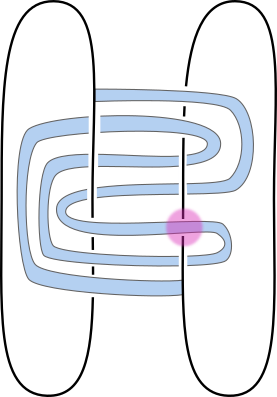}
    \end{minipage}
    \hfill
    \begin{minipage}{0.5\textwidth}
    \centering
    \includegraphics[width=.7\textwidth]{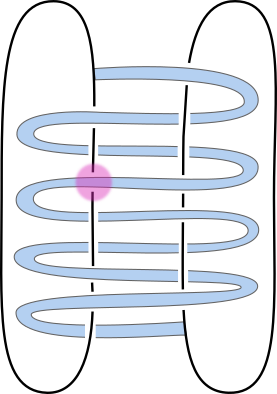}
    \end{minipage}

    \begin{minipage}[t]{0.5\textwidth}
    \caption{A ribbon disk for the knot $9_{41}$ with undisking number one.}
    \label{fig:kawauchidisks1}
    \end{minipage}
    \hfill
    \begin{minipage}[t]{0.5\textwidth}
    \caption{A ribbon disk for the knot $10_{75}$ with undisking number one.}
    \label{fig:kawauchidisks2}
    \end{minipage}
\end{figure}

\begin{lemma}\label{lem:kawauchidisks}
Of the ribbon disks given for ribbon knots in \cite[F.5]{Kaw96book}, the following have undisking number equal to one.
\[\{6_1,8_8,8_9,8_{20}, 9_{27}, 9_{41}, 9_{46}, 10_{48},10_{75}, 10_{87},10_{129},10_{137}, 10_{140},10_{153},10_{155}\}\]
\end{lemma}

\begin{proof}
This fact can be verified by checking that each banded unlink diagram illustrated in \cite[F.5]{Kaw96book} for each ribbon disk bounded by the knots specified above has (at least) one band pass move taking it to a trivial banded unlink diagram. Figures \ref{fig:kawauchidisks1} and \ref{fig:kawauchidisks2} illustrate two such banded unlink diagrams in Kawauchi's table for ribbon disks bounded by $9_{41}$ and $10_{75}$. A band move on the highlighted crossing in each diagram gives a trivial banded unlink diagram.
\end{proof}

One can also check that the remaining ribbon disks have undisking number at most two, by finding two band pass moves producing a trivial diagram. 

\begin{lemma}
The ribbon disks for the knots 
$$\{10_3, 10_{22}, 10_{35}, 10_{42}, 10_{99}, 10_{123}\}$$
given in \cite[F.5]{Kaw96book} have undisking number at most two. 
\end{lemma}

One easy way to produce ribbon disks is to ``spin'' a knot. If $K\subset S^3$ is a knot, let $D_K\subset B^4$ be the ribbon disk $(B^3\times I,K^\circ\times I)$ for the connected sum $K\#-K$. 

\begin{lemma}\label{lem:undiskingnumtwistknot}
Suppose that $T_n$ is the twist knot with $n$ half twists, i.e. the $n$-twisted Whitehead double of the unknot. Then $\delta(D_{T_n})=1$. 
\end{lemma}

\begin{proof}
A banded unlink diagram for the disk $D_{T_{n}}$ in which a single band pass trivializes the diagram is given in Figure \ref{fig:twistdisks} below. 
\end{proof}

\begin{figure}[ht]
\includegraphics[height=160pt]{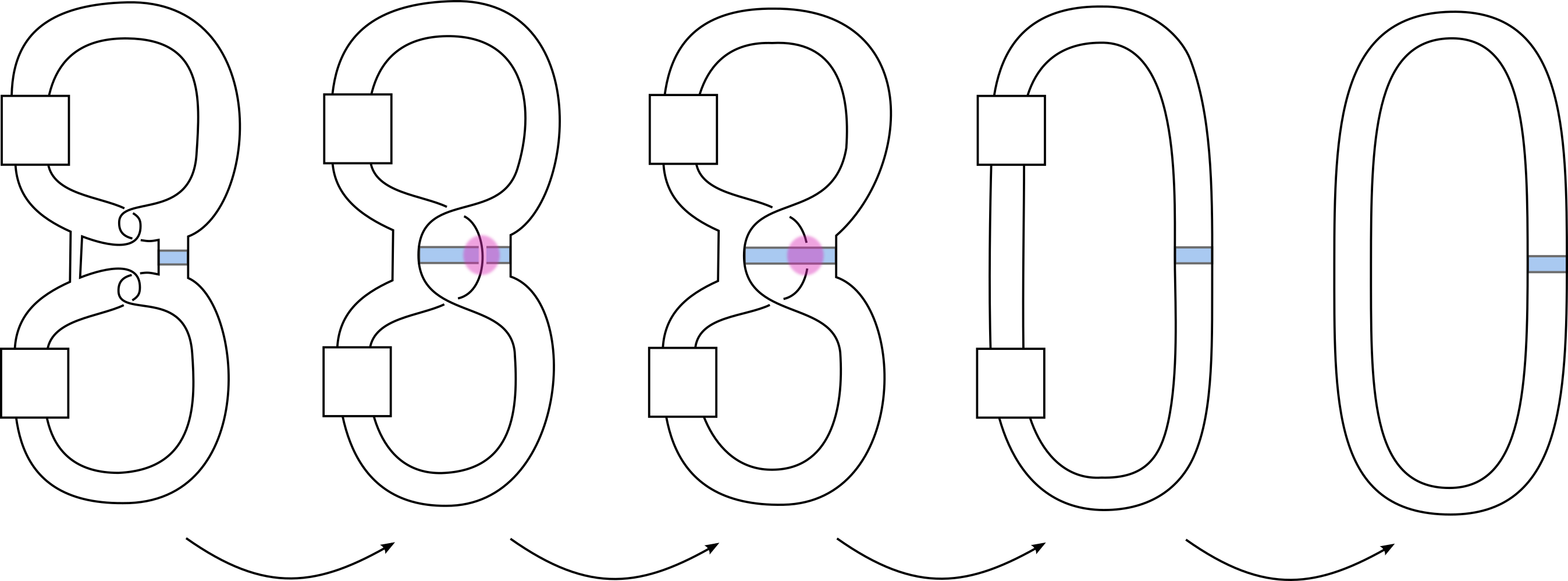}
\put(-425,122){\small $n$}
\put(-337,122){\small $n$}
\put(-247,122){\small $n$}
\put(-157,122){\small $n$}
\put(-430,52){\small $-n$}
\put(-341,52){\small $-n$}
\put(-251,52){\small $-n$}
\put(-161,52){\small $-n$}
\put(-367,-10){\small{Isotopy}}
\put(-284,-10){\small {Band pass}}
\put(-188,-10){\small{Isotopy}} 
\put(-91,-10){\small{Isotopy}}
\caption{A banded unlink diagram for $D_{T_{n}}$ (left), using the convention that the integer $n$ records the number of signed half twists between the vertical strands as they ascend through the boxed region. The highlighted band pass move takes this diagram for $D_{T_{n}}$ to the trivial banded unlink diagram on the right.}
\label{fig:twistdisks}
\end{figure}
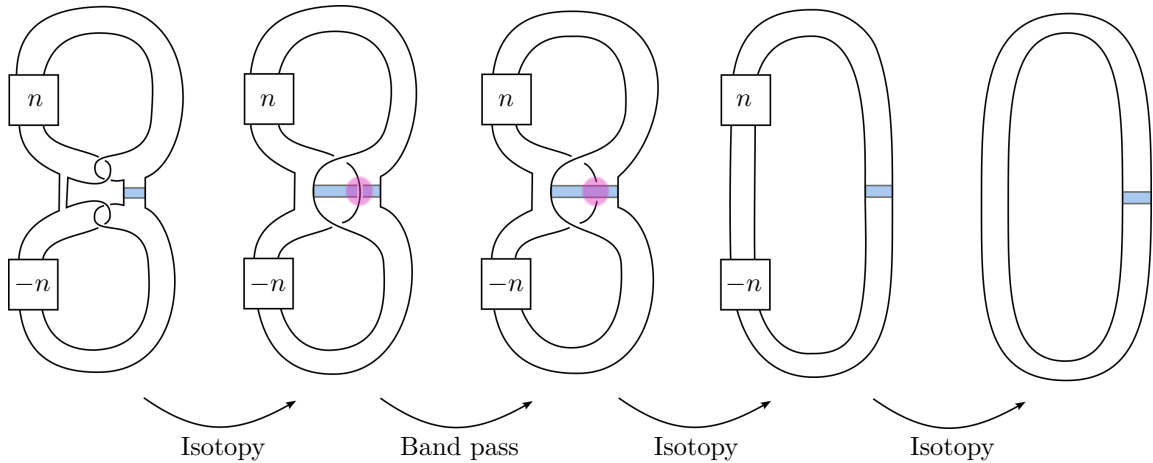

\begin{definition} \label{def:triangle}
Let $\mathcal D$ be a banded unlink diagram for a ribbon disk, as in Definition \ref{def:bud}. As in Remark \ref{rem:3handlespheres}, let $D_1, \dots, D_n$ denote spanning disks for the components of the unlink $\mathcal U$. The bands $\beta_1, \dots, \beta_{n-1}$ of $\mathcal D$ are assumed to intersect these disks transversely in ribbon arcs. The diagram $\mathcal D$ is called \bit{triangular} if $n=3$ and
\begin{enumerate}
    \item only the band $\beta_1$ intersects the disk $D_3$,
    \item only the band $\beta_2$ intersects the disk $D_1$, and
    \item the band $\beta_2$ intersects $D_1$ exactly once.
\end{enumerate}
Either of the bands $\beta_1$ and $\beta_2$ is free to intersect the middle disk $D_2$. 
\end{definition}

\begin{lemma}\label{lem:uppertriangular}
Let $R\subset B^4$ be a ribbon disk. Then $\delta(R)=1$ if and only if $R$ corresponds to a triangular banded unlink diagram.
\end{lemma}

\begin{proof}
If $R$ corresponds to a triangular banded unlink diagram, then it is easy to check that $\delta(R)=1$. Conversely, suppose that $\delta(R)=1$. Choose a banded unlink diagram $\mathcal D_0$ for $R$ realizing the fact that $\delta(R)=1$, i.e. such that one band pass move between a minimum $U_i$ and band $\beta_j$ takes $\mathcal D_0$ to a trivial diagram. We will modify $\mathcal{D}_0$ using the moves from Definition \ref{def:moves} to produce a (new) triangular banded unlink diagram. Recall that each move preserves the isotopy class of the corresponding disk $R$. Figures \ref{fig:trianglebud1} and  \ref{fig:trianglebud2} illustrate the steps detailed below. 

First, note that since performing a single band pass move of $\beta_j$ yields a trivial banded unlink diagram, the band $\beta_j$ must connect the $j^{th}$ and $(j+1)^{st}$ components of the unlink $\mathcal U$. The bands of $\mathcal D_0$ can therefore be isotoped rel boundary so that $\beta_j$ is the only band passing through any component of $\mathcal U$, as in the top diagram of Figure \ref{fig:trianglebud1}\footnote{Or, more rigorously, so that only the band $\beta_j$ intersects the standard spanning disks $D_1, \dots, D_n$ for the unlink $\mathcal U$ from Remark \ref{rem:3handlespheres}.}. 

Furthermore, we can achieve this without disturbing the crossing between $U_i$ and $\beta_j$ where the band pass move occurs, by isotoping all other bands rel boundary to ``standard" position and dragging $\beta_j$ along only when necessary. We can then perform band swims of $\beta_j$ over the other bands to arrange that $\beta_j$ passes only through $U_j$ and $U_{j+1}$, as shown in Step (1) of Figure \ref{fig:trianglebud1}. 

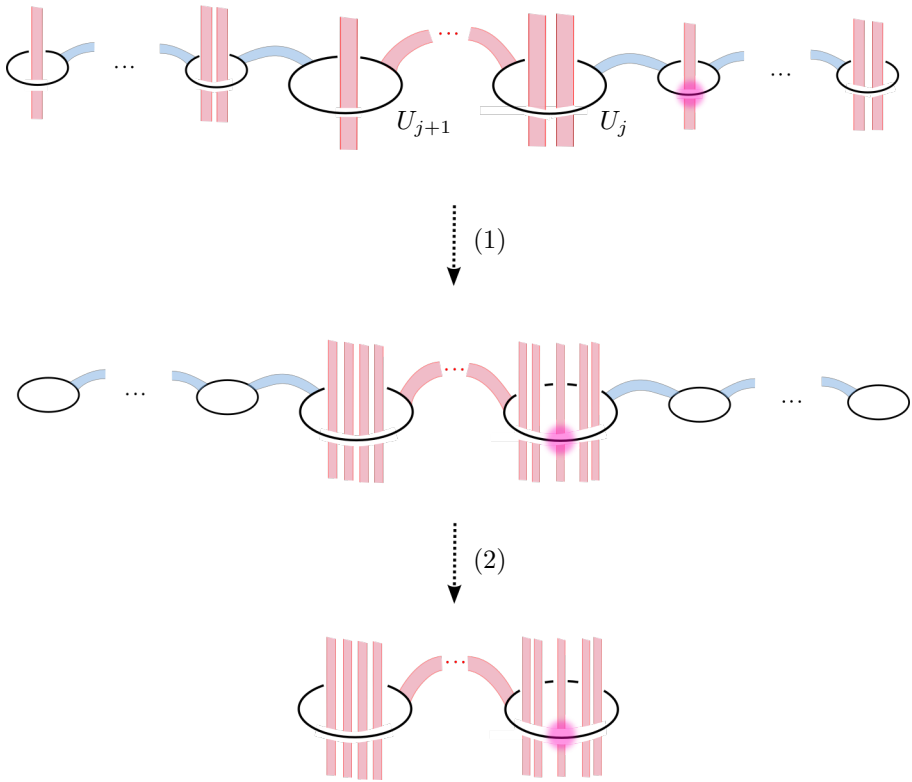
\begin{figure}[ht]
\includegraphics[height=300pt]{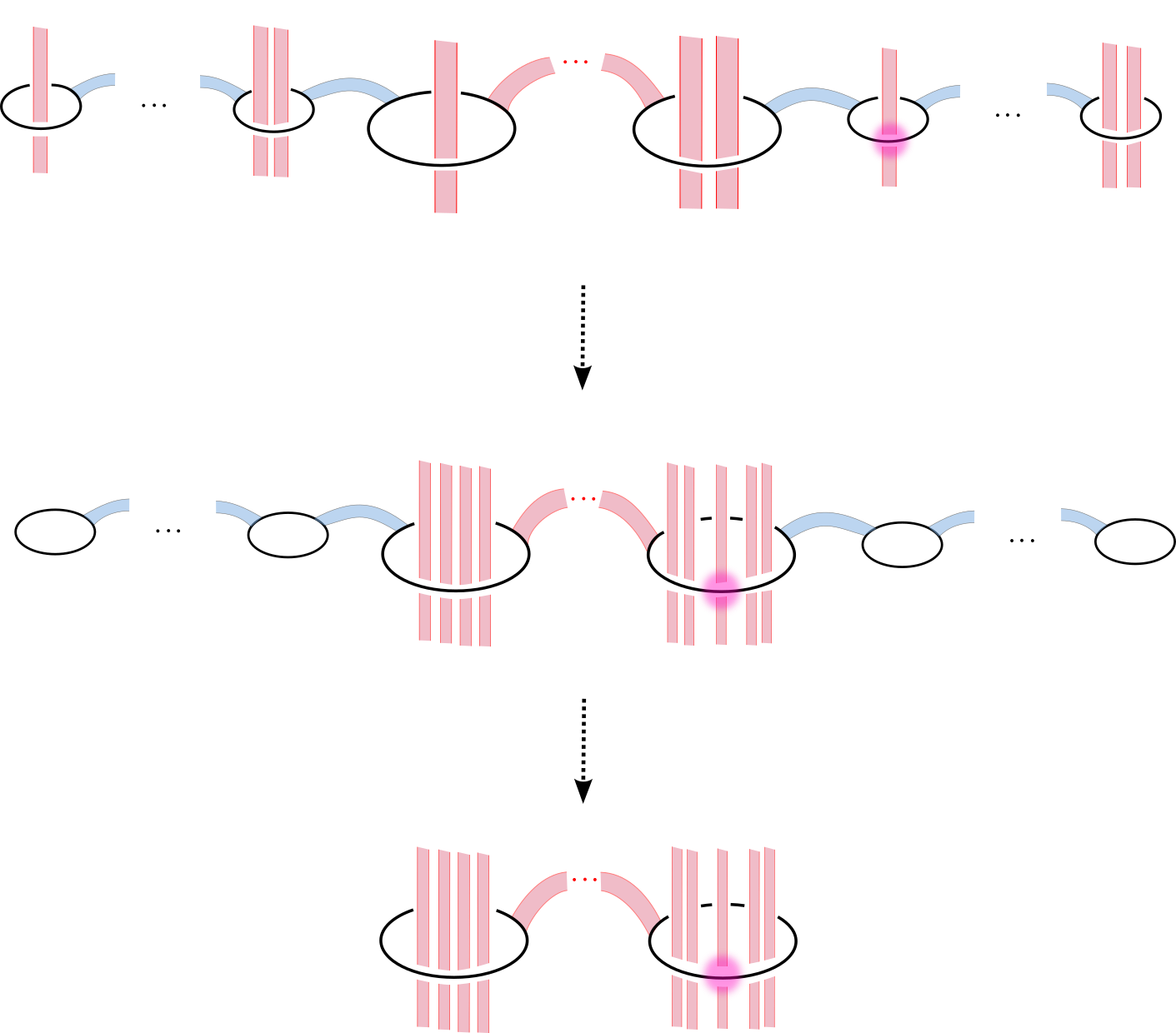}
\put(-165,201){\small $(1)$}
\put(-117,246){\small $U_j$}
\put(-194,246){\small $U_{j+1}$}
\put(-165,80){\small $(2)$}
\caption{The first two steps from the proof of Lemma \ref{lem:uppertriangular}, with the band $\beta_j$ in red. The band pass move that gives a trivial diagram is highlighted in each step. {\bf Step (1):} Band swim $\beta_j$ over other bands so that it only passes through $U_j$ and $U_{j+1}$. {\bf Step (2):} Cancel all but two components of $\mathcal U$.}
\label{fig:trianglebud1}
\end{figure}

The remaining components $U_k$ with $k \not \in \{j, j+1\}$ can now be cancelled with an adjacent band, as in Definition \ref{def:moves} and Figure \ref{fig:moves}. This produces the banded unlink diagram $\mathcal D_1$ both at the bottom of Figure \ref{fig:trianglebud1} and top of Figure \ref{fig:trianglebud2}. Note that we have now produced a ribbon embedding of $R$ with only two minima. By a slight abuse of notation, we continue to refer to the band as $\beta_j$ and the two components of the unlink as $U_j$ and $U_{j+1}$. Without loss of generality, we assume that performing a band pass between $\beta_j$ and the component $U_j$ produces a trivial banded unlink diagram from $\mathcal D_1$. 

Finally, introduce a new component $U_*$ to the unlink $\mathcal U$, connected to $U_j$ by a new band $\beta_*$. Perform a single band swim of $\beta_j$ through $\beta_*$, as in Step (3) of Figure \ref{fig:trianglebud2}, to produce a new banded unlink diagram $\mathcal D_2$ for $R$ with the property that a single band pass of $\beta_j$ through $U_*$ produces the trivial diagram. In particular, deleting $\beta_*$ and $U_*$ from $\mathcal{D}_2$ gives a trivial diagram. Therefore, $\beta_j$ can be isotoped rel boundary to look trivial with respect to $U_j$ and $U_{j+1}$, dragging along the unknot $U_*$ and the band $\beta_*$ in the process. This isotopy of the diagram replaces all of the crossings between $U_{j+1}$ and $\beta_j$ with crossings between $U_{j+1}$ and $\beta_*$. Thus, the conditions of Definition \ref{def:triangle} are satisfied, and we have produced a triangular banded unlink diagram for $R$. 

\begin{figure}[ht]
\includegraphics[height=300pt]{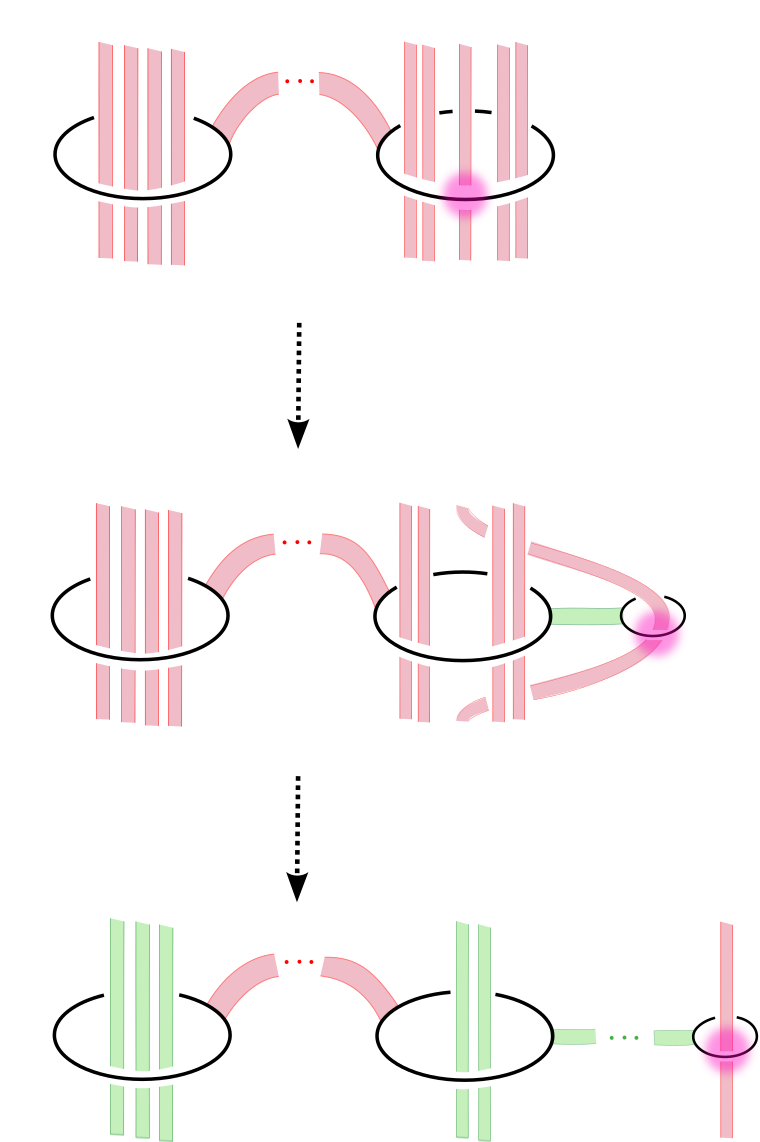}
\put(-117,199){\small $(3)$}
\put(-117,79){\small $(4)$}
\caption{The final two steps from the proof of Lemma \ref{lem:uppertriangular}, with the band $\beta_j$ in red and the band $\beta_*$ in green. The band pass move that gives a trivial diagram is highlighted in each step. {\bf Step (3):} Introduce $U_*$ and $\beta_*$ to the diagram. Band swim $\beta_j$ over $\beta_*$. {\bf Step (4):} Further isotopy of diagram, dragging $U_*$ along with $\beta_j$.}
\label{fig:trianglebud2}
\end{figure}
\end{proof}

\begin{remark}
The upcoming proof of Theorem \ref{thm:main} relies heavily on the fact that the ribbon disks we consider admit triangular banded unlink diagrams. If $\delta(R)>1$, it is not obvious whether $R$ admits a similar ``nice'' kind of banded unlink diagram (in the combinatorial sense) or if Theorem \ref{thm:main} can be extended to a larger family of spheres.
\end{remark}

\section{Five-dimensional handlebodies}\label{sec:5d}

\subsection{Products of Gluck twists}

Let $S$ be a $2$-sphere in $S^4$ given by a closed banded unlink diagram $\mathcal D$. Recall that $\mathcal{D}$ consists of a collection $\beta$ of $k$ bands attached to an $n$ component unlink $\mathcal U$, whose resolution is an $m$ component unlink $L_\beta$. The standard handlebody decomposition of the punctured Gluck twist $\Sigma_S^\circ$ induced by $\mathcal D$ consists of a 0-handle, $m$ $1$-handles, $k$ $2$-handles, and $n-1$ $3$-handles, as well an extra ``Gluck'' $2$-handle.

\begin{definition}\label{def:std5handlebody}
Let $D$ be a banded unlink diagram for a sphere $S\subset S^4$. The \bit{standard handle decomposition of the product $W_S=\Sigma_S^\circ \times I$} induced by the diagram $\mathcal D$ is the handle structure obtained by taking the product of the standard handle decomposition of $\Sigma_S^\circ$ with the interval $I$. This $5$-dimensional homotopy ball is built by attaching $m$ $1$-handles, $k$ $2$-handles, $n-1$ $3$-handles, and an extra ``Gluck $2$-handle" to the $5$-ball.
\end{definition}

\begin{remark} \label{rem:equiv}
The product $\Sigma_S^\circ \times I$ is diffeomorphic to $B^5$ if and only if $\Sigma_{S}\cup -\Sigma_S$ is diffeomorphic to the standard $4$-sphere. Indeed, if $\Sigma_S^\circ\times I=B^5$, then
\[S^4=\partial(B^5)=\partial(\Sigma_S^\circ\times I)=\Sigma_S\cup-\Sigma_S.\]
Conversely, if $\Sigma_S\cup-\Sigma_S\cong S^4$, then $\Sigma_S^\circ\times I$ is a homotopy 5-ball with boundary $S^4$, which is diffeomorphic to $B^5$ by \cite{Sma62}.   Moreover, we also have $\Sigma_S\cup -\Sigma_S=\Sigma_{S\# -S}$, so these conditions are equivalent to the statement that $\Sigma_{S\#-S}\cong S^4$. 
\end{remark}

In the next section, our goal will be to show that $\Sigma_S^\circ\times I\cong B^5$. We will simplify its handle structure by cancelling both 1/2- and 2/3- pairs. While the 1/2- pairs can be dealt with algebraically (Section \ref{subsec:5d1and2}), the 2/3 pairs must be manipulated to cancel geometrically (Section \ref{subsec:5d2and3}). We outline both cases separately in the following two subsections.

\subsection{Cancelling 5-dimensional 1/2-pairs} \label{subsec:5d1and2}

A handlebody structure on either a $4$- or $5$-manifold $X$ with a single $0$-handle induces a presentation of its fundamental group 
$$\pi_1(X) = \langle x_1, \dots, x_m ~|~ r_1, \dots, r_k \rangle,$$
with base point (suppressed in our notation) in the $0$-handle. Each generator $x_i$ is represented by a based loop passing once over the $i^{th}$ $1$-handle, while each relation $r_i$ is represented by a based loop running once around the attaching circle of the $i^{th}$ $2$-handle. Both the relations and generators are unique only up to conjugation and inversion, since they are determined by a choice of orientation and a ``whisker" from the attaching region of the 1- or 2-handle to the base point. 

\begin{definition}\label{def:ac}
Let $G$ be a group with presentation $\langle x_1, \dots, x_m ~|~ r_1, \dots, r_k \rangle$. The following algebraic manipulations are referred to as \bit{Andrews-Curtis (AC) moves} on $G$. 

\begin{enumerate}
    \item Replace a relation $r_i$ by a product $r_i r_j$ for any $i \not = j$.
    \item Replace a relation $r_i$ by its inverse.
    \item Replace a relation $r_i$ by its conjugate.
\end{enumerate}

A balanced presentation of the trivial group (with $m=k$) is called \bit{Andrews-Curtis trivializable} if it can be reduced to a \bit{trivial} presentation $\langle x_1, \dots, x_m ~|~ x_1, \dots, x_m \rangle$ through a finite sequence of AC moves. If $m<k$, we generalize this notion and refer to the presentation as \bit{Andrews-Curtis $m$-trivializable} if it can be reduced via a finite sequence of AC moves to an \bit{$m$-trivial} presentation with relations $r_i=x_i$ for $i=1, \dots, m$, and \bit{excess} relations $r_{m+1}, \dots, r_k$.     
\end{definition}

It is well-known that the AC-triviality of a presentation is preserved by replacing a generator by its product with another generator and modifying the relations accordingly (or more generally, replacing the generators by their image under any automorphism of the free group $F_m$). Whether all balanced presentations of the trivial group can be trivialized using AC moves is known as the \emph{Andrews-Curtis conjecture}. The (a priori weaker) stable version of this conjecture, which allows the addition or deletion of a generator $x$ and trivial relation $r=x$ is also open. For example, it is unknown whether the following presentations of the trivial group are stably Andrews-Curtis trivial for $n \geq 3$ (see Remark \ref{rem:gompfex}).
\[\langle x,y\mid xyx=yxy, x^{n+1}=y^{n}\rangle\] 

In our context, Andrews-Curtis moves will be used to keep track of the presentation of $\pi_1(X)$ induced by handle structures before and after handle slides. We will frequently use the fact that sliding the $i^{th}$ $2$-handle over the $j^{th}$ $2$-handle of $X$ replaces the relation $r_i$ by a product $r_i w_j$ where $w_j \in \pi_1(X)$ is some conjugate of $r_j$ that depends on the path used to perform the 2-handle slide. An example of a distinct pair of slides causing this indeterminacy is illustrated in Figure \ref{fig:slide}; for more details we refer the reader to \cite{Gom91}, \cite[\S 1]{MelSch21}, and \cite{CFHS96}.

\begin{figure}[ht]
\includegraphics[height=120pt]{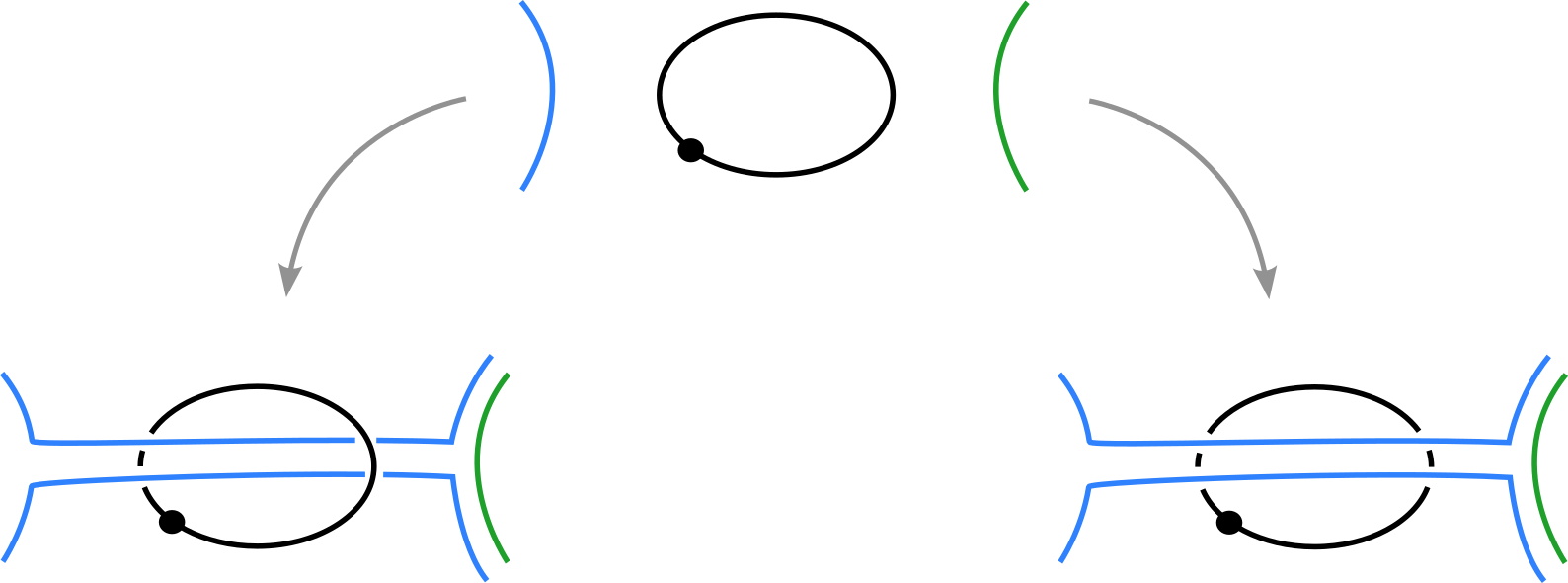}
\put(-300,84){\small Slide}
\put(-50,84){\small Slide}
\caption{Handle slides of one $2$-handle over another can be done along different paths in the $1$-handles. A choice of path determines the relation induced by the new $2$-handle, in the new induced presentation of the fundamental group.}
\label{fig:slide}
\end{figure}

We present the following well-known fact as a lemma. 

\begin{lemma}\label{lem:ACproduct}
Suppose that a $5$-manifold $W$ has a handlebody structure consisting of an equal number of $1$ and $2$-handles attached to the $5$-ball. If the induced presentation of $\pi_1(W)$ is AC trivial, then $W$ is diffeomorphic to the standard $5$-ball.
\end{lemma}

\begin{proof}
Since the presentation of $\pi_1(W)$ induced by the given handlebody structure is AC trivial, the induced presentation can be arranged to be trivial (as in Definition \ref{def:ac}) by performing a sequence of $1$- and $2$-handle slides. 

Let $W^{(1)}$ denote the submanifold of $W$ consisting of only the $1$-handles of $W$ attached to the $5$-ball. Note that in general, the presentation of $\pi_1(W)$ prescribes the homotopy classes of the 2-handle attaching curves in $\partial W^{(1)}$. Because this presentation is trivial, each $2$-handle algebraically -- in fact, \emph{geometrically} -- cancels a $1$-handle, since homotopy implies isotopy for curves in $4$-manifolds. Cancelling all of the $1$- and $2$-handle pairs leaves $B^5$, as desired. 
\end{proof}

\begin{lemma} \label{lem:budac}
Let $\mathcal D$ be any closed banded unlink diagram for an embedded sphere $S \subset S^4$, with unlink $L_\beta$ consisting of $m$ components. The standard handlebody structure on the Gluck twist $\Sigma_S$ corresponding to $\mathcal D$ induces an $m$-trivial presentation of $\pi_1(\Sigma_S) = 1$.
\end{lemma} 

\begin{proof}
Recall that in the standard handlebody structure corresponding to $\mathcal D$ for the Gluck twist $\Sigma_S$, the Gluck 2-handle is attached along the meridian to the first $1$-handle $\mu_1$ (see Definition \ref{def:stdhandle}). Thus, the presentation of $\pi_1(\Sigma_S)$ that this handle decomposition induces includes $x_1$ as a relation. 

Meanwhile, each of the fission $2$-handles $h_{i+n-1}$ (for $1 \leq i \leq m$) gives a relation of the form $x_{i+1}=\omega_i x_i \omega_i^{-1}$ for some $\omega_i \in F_m$, since it is attached around the fission band $\beta_{i+n-1}$ that splits the components $\mu_i$ and $\mu_{i+1}$ of the unlink $L_\beta$. It is easy to check that the resulting presentation is AC $m$-trivial. 
\end{proof}

\begin{remark}
While it does guarantee that the 1- and 2-handles of $\Sigma_S^\circ\times I$ can be cancelled, Lemma \ref{lem:budac} does not immediately imply that $\Sigma_S^\circ\times I$ is diffeomorphic to the 5-ball, since our preferred handle decomposition of this 5-manifold generally has additional 2- and 3-handles. In the next section, we will try to cancel the higher index handles. In general, we will have to be careful, so as not to affect the AC-triviality of the 1- and 2-handles. 
\end{remark}

\subsection{Cancelling 5-dimensional 2/3-pairs}\label{subsec:5d2and3} 

\begin{definition}\label{def:ascdesspheres}
Let $L_S$ be the 4-dimensional level set of $W_S$ after the (5-dimensional) 2-handles are added but before the 3-handles are attached. There are two distinguished sets of disjointly embedded $2$-spheres in $L_S$, namely:
\begin{enumerate}
    \item The descending $2$-spheres $\mathcal R = \{R_1, \dots, R_{n-1}\}$ of the $n-1$ $3$-handles, and
    \item the ascending $2$-spheres $\mathcal G = \{G_1,\dots,G_k,G_*\}$ of the $k+1$ 2-handles, where $G_1,\dots,\allowbreak G_k$ are the ascending spheres of the first $k$ 2-handles, and $G_*$ is the ascending sphere of the Gluck 2-handle.
\end{enumerate}
\end{definition}

Note that the spheres in $\mathcal R$ and $\mathcal G$ may intersect (transversely). To cancel the 2- and 3-handles, we must understand these intersections.

\begin{definition}\label{def:geodual}
Suppose that $\mathcal{A}=\{A_1,\dots,A_k\}$ and $\mathcal{B}=\{B_1,\dots,B_k\}$ are two collections of embedded 2-spheres in a 4-manifold $X$ that intersect each other transversely. The sets $\mathcal{A}$ and $\mathcal{B}$ are called \bit{geometrically dual} if (up to relabelling), the intersection $A_i\pitchfork B_i$ is a single point, and $A_i \pitchfork B_j$ is empty for all $j \not =i$. The sets of spheres are called \bit{algebraically dual} if the algebraic intersection numbers satisfy $A_i \cdot B_i=\pm 1$ and $A_i \cdot B_j =0$ for all $j \not = i$.  
\end{definition}

\begin{lemma}
Suppose that a 5-manifold $W$ has a handlebody structure consisting only of an equal number of 2- and 3-handles attached to the 5-ball. Let $\mathcal{G}$ and $\mathcal{R}$ denote the collections of ascending and descending spheres for the 2- and 3-handles of $W$, respectively, viewed in a level set in after all 2-handles have been attached, but before any 3-handles have been attached. If $\mathcal{G}$ and $\mathcal{R}$ are geometrically dual, then $W$ is diffeomorphic to the standard 5-ball. 
\end{lemma}

\begin{proof}
Since $\mathcal{G}$ and $\mathcal{R}$ are geometrically dual, the 2- and 3-handles partition into pairs that cancel geometrically. This leaves $B^5$, as desired. 
\end{proof}

On the other hand, any contractible 5-manifold $W$ built from $B^5$ using only 2- and 3-handles is not clearly diffeomorphic to $B^5$. Indeed, many $5$-dimensional $h$-cobordisms built using algebraically cancelling $2$ and $3$-handles are known not to be diffeomorphic to a product, following work of \cite{Don87}. However, given such a $5$-manifold, a variety of techniques exist to modify (and hopefully, \emph{simplify}) the intersection pattern between the ascending and descending spheres $\mathcal{G}$ and $\mathcal{R}$. Such a non-trivial intersection pattern is illustrated in Figure \ref{fig:4dlevelset}; we will address this example in \S \ref{sec:mainresults}.

\begin{figure}[ht]
\includegraphics[height=190pt]{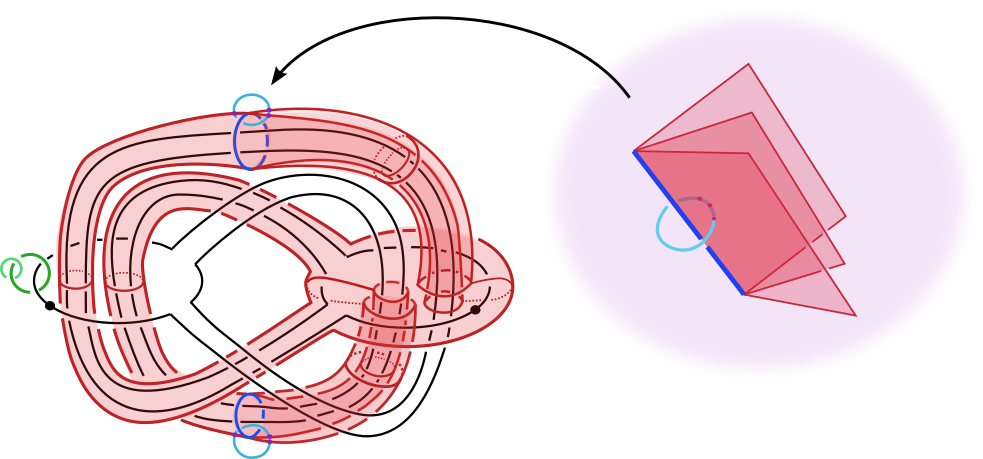}
\caption{An illustration of the level set $L_S$, where $S$ is the 2-twist spun trefoil from Figure \ref{fig:foxbud}. Note that the ascending spheres for the 2-handles (light blue) intersect the descending sphere for the 3-handle (red) in three points (but only once algebraically).}
\label{fig:4dlevelset}
\end{figure}

Towards this goal, suppose that $R$ and $G$ are smoothly embedded surfaces in a 4-manifold $X^4$ intersecting transversely in oppositely signed points $p$ and $q$, as in the left-hand side of Figure \ref{fig:wm}.

\begin{definition} \label{def:wm1}
Suppose $W \subset X^4$ is an immersed disk such that

\begin{enumerate}
    \item $G \pitchfork \partial W$ is an embedded arc $\gamma$ with endpoints $p$ and $q$,
    \item $R \pitchfork \partial W$ is an embedded arc $\rho$ with endpoints $p$ and $q$.
\end{enumerate}

\begin{figure}[ht]
\includegraphics[height=95pt]{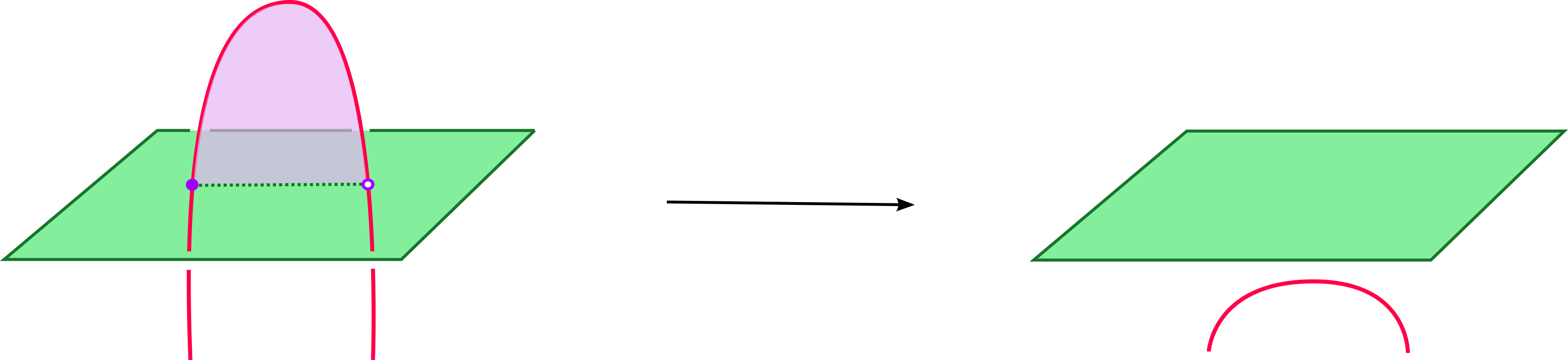}
\put(-345,68){\small $W$}
\put(-344,100){\small $\rho$}
\put(-345,37){\small $\gamma$}
\put(-311,-2){\small $R$}
\put(-410,15){\small $G$}
\put(-311,45){\small $p$}
\put(-373,45){\small $q$}
\put(-311,-2){\small $R$}
\put(-136,15){\small $G$}
\put(-37,-1){\small $R$}
\put(-230,28){\small Homotopy}
\caption{The local picture of a Whitney homotopy, as in Definition \ref{def:wm2}, removing pairs of intersection points between surfaces $R$ and $G$.}
\label{fig:wm}
\end{figure}
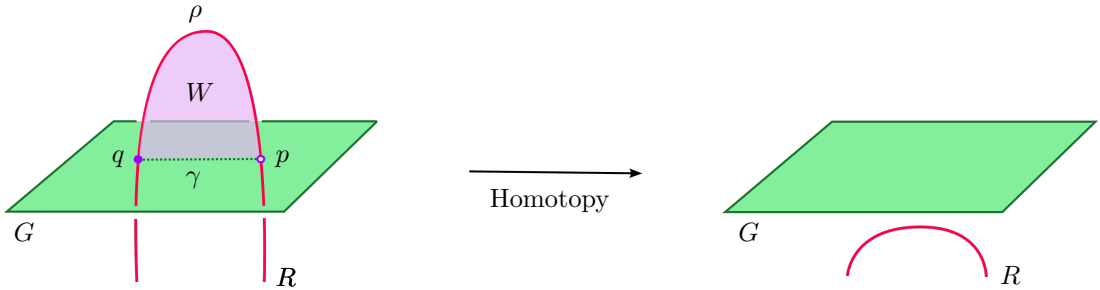

In this case, the disk $W$ is called a \bit{pre-Whitney disk} pairing the intersection points $p$ and $q$, and its boundary $\partial W=\rho\cup\gamma$) is called a \bit{Whitney circle}. 
\end{definition}

The normal $D^2$-bundle $\nu_W$ of a pre-Whitney disk $W$ has a unique trivialization up to isotopy.
Restricting this trivialization to $\nu_{\partial W}$ induces a framing on $\partial W$. The disk bundle $\nu_{\partial W}$ can also be identified with a smoothing of a trivial bundle $S^1 \times (I \times I)$ where each $I \times I$ fiber satisfies the following: 

\begin{enumerate}
    \item $I \times \pt$ is parallel to $G$ along $\gamma$, 
    \item $\pt \times I$ is parallel to $G$ along $\rho$,
    \item $\pt \times I$ is parallel to $R$ along $\gamma$, and 
    \item $I \times \pt$ is parallel to $R$ along $\rho$.
    
\end{enumerate}

This gives a second natural framing of $\partial W$, obtained from the immersion $R \pitchfork G$, referred to as the \bit{Whitney framing} of the Whitney circle. Refer to \cite{fqbook} for more details. 

\begin{definition} \label{def:wm2}
Let $W$ be an embedded pre-Whitney disk whose interior is disjoint from both $R$ and $G$. If the disk framing on the Whitney circle $\partial W$ induced by $W$ is isotopic to the Whitney framing induced by the surfaces $R$ and $G$, then $W$ is called a \bit{Whitney disk} pairing the intersection points $p$ and $q$. 
\end{definition}

If $W$ is an honest Whitney disk, it admits the local model shown on the left-hand side of Figure \ref{fig:wm}. This gives a regular homotopy of $R \cup G$ to the right-hand side of Figure \ref{fig:wm}, in which the pair of intersection points $p$ and $q$ have been removed. Whitney developed this local homotopy, now called a \bit{Whitney move} along $W$, to remove pairs of algebraically cancelling double points \cite{Whi44}.

\section{Main results}\label{sec:mainresults}

We are now ready to prove the main results of this paper. As a warm-up, we consider the case in which a 2-sphere $S\subset S^4$ admits an embedding with only \emph{one} maximum.\footnote{As a special case of the unknotting conjecture, spheres of this form were first conjectured to be smoothly unknotted by Suzuki \cite{Suz76}. This conjecture remains open, although it is now known by Freedman's work \cite{Fre82sc} that such spheres are topologically unknotted.} 

\begin{theorem}\label{thm:suzukispheres}
Let $S\subset S^4$ be an embedded 2-sphere with exactly one maximum with respect to the radial height function on $S^4$. Then $\Sigma_{S \cs -S}$ is diffeomorphic to $S^4$. 
\end{theorem}

\begin{proof}
Since the given embedding of the sphere $S$ has only one maximum, it may be described by a closed banded unlink diagram $\mathcal D$ with the property that the link $L_\beta$ has only one component. Therefore, the standard handle decomposition of the Gluck twist $\Sigma_S$ with respect to the diagram $\mathcal D$ is built from the $4$-ball using only $1$ and $2$-handles, capped off by a single $4$-handle. 

By Lemma \ref{lem:budac}, the induced presentation of the trivial group $\pi_1(\Sigma_S^\circ \times I) = 1$ given by this handlebody structure is AC trivial. Thus, it follows from Lemma \ref{lem:ACproduct} that the product $\Sigma_S^\circ \times I$ is diffeomorphic to $B^5$, or equivalently by Remark \ref{rem:equiv}, that $\Sigma_{S \cs -S}$ is diffeomorphic to $S^4$. 
\end{proof}

\begin{remark}
In this case, $S\#-S$ is the double of a ribbon disk, and so is a ribbon 2-knot. Consequently, it also follows from Gluck's original work \cite{Glu62} that $\Sigma_{S\#-S}$ is standard. 
\end{remark}

We will now consider the possibility that $S\subset S^4$ has more than one maximum. In this case, the standard handlebody decomposition of $\Sigma_S^\circ\times I$ has 3-handles that must be cancelled. Provided that at least one ribbon hemisphere of $S$ is sufficiently simple, we obtain the same conclusion as the previous theorem.  

\begin{theorem}\label{thm:main}
Suppose that a 2-sphere $S$ admits a decomposition into two ribbon disks, one of which has undisking number equal to one. Then $\Sigma_S^\circ\times I$ is diffeomorphic to $B^5$.
\end{theorem}

\proof  Let $S=D_1\cup D_2$ be a decomposition of $S$ into two ribbon disks, where $\delta(D_1)=1$. By Lemma \ref{lem:uppertriangular}, there is a triangular banded unlink diagram for $D_1$ with fusion bands $b_1$ and $b_2$. Adding fission bands $f_1, \dots, f_\ell$ corresponding to the saddles of $D_2$, as shown in Figure \ref{fig:t1}, gives a schematic for a banded unlink diagram $\mathcal D$ for the sphere $S$. 

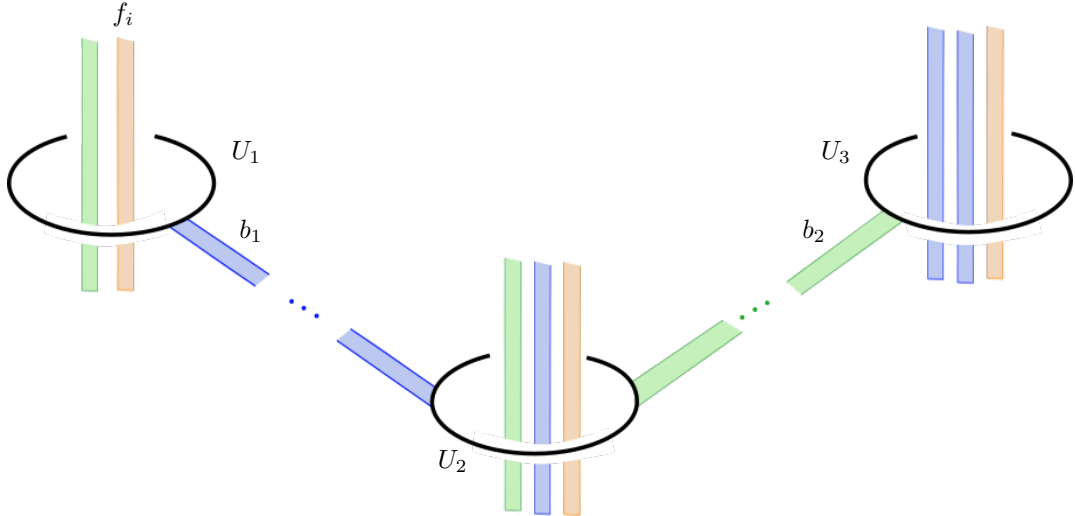
\begin{figure}[ht]
\includegraphics[height=185pt]{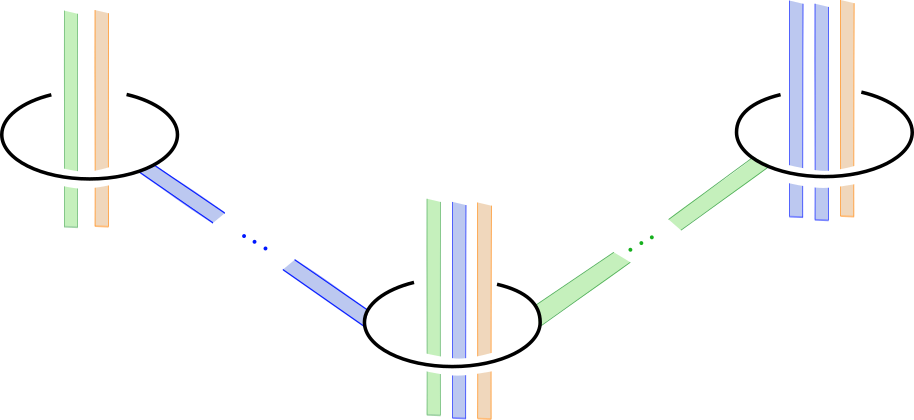}
\put(-315,105){\small $b_1$}
\put(-102,105){\small $b_2$}
\put(-318,135){\small $U_1$}
\put(-363,187){\small $f_i$}
\put(-240,18){\small $U_2$}
\put(-95,135){\small $U_3$}
\caption{An example of a banded unlink diagram $\mathcal D$ for a $2$-sphere $S$. The fusion bands $b_1$ and $b_2$ are drawn in blue and green (resp.), while the fission bands are illustrated in yellow (attaching regions omitted). By construction, performing a band pass move of $b_2$ through $U_1$ trivializes the banded unlink diagram for $D_1$.}
\label{fig:t1}
\end{figure}

The disk $D_1$ has three minima, and so the standard handle decomposition (see Definition \ref{def:std5handlebody}) of the $5$-manifold $W_S =\Sigma_S^\circ\times I$ has two 3-handles, with attaching spheres $R_1$ and $R_2$. Since the banded unlink diagram for $D_1$ is triangular, there are two 5-dimensional fusion 2-handles with ascending 2-spheres $G_1$ and $G_2$, together with the Gluck 2-handle, whose ascending sphere we label $G_*$. We choose to attach the Gluck $2$-handle along the meridian of the $1$-handle that achieves the set-up shown in Figures and \ref {fig:t2} and \ref{fig:t3}. We will label the ascending 2-spheres corresponding to the fission bands by $F_1,\dots,F_\ell$.

Figure \ref{fig:t2} illustrates the fact that in this special case, each pair $(R_1,G_1)$ and $(R_2,G_2)$ is already geometrically dual. However, in order to cancel \emph{both} $2/3$-handle pairs simultaneously, we will perform slides and isotopies of both sets of spheres to arrange for the set $\{R_1, R_2\}$ to be dual to the set $\{G_1, G_2\}$. We will argue that the remaining 1- and 2-handles cancel via an Andrews-Curtis argument. 

\begin{figure}[ht]
\includegraphics[height=200pt]{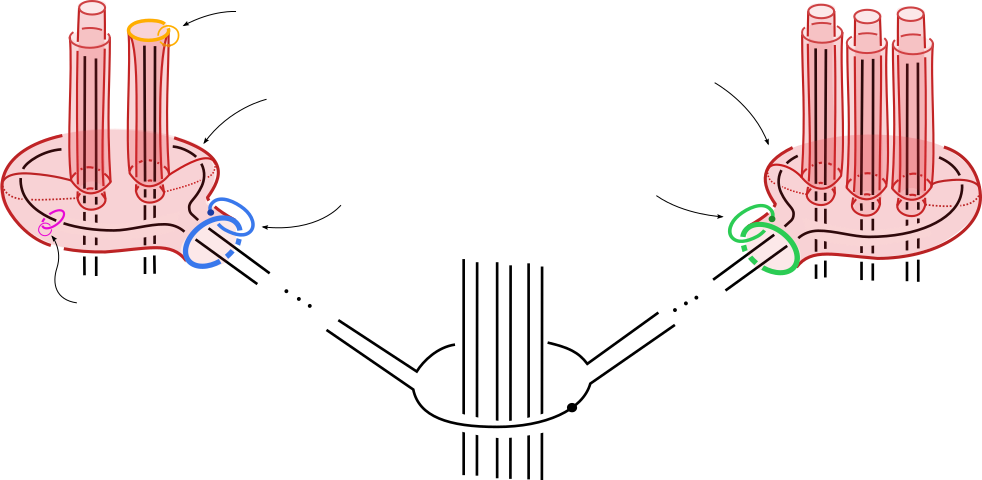}
\put(-372,70){\small $G_*$}
\put(-150,118){\small $G_2$}
\put(-265,115){\small $G_1$}
\put(-126,166){\small $R_2$}
\put(-307,193){\small $F_i$}
\put(-295,157)
{\small $R_1$}
\caption{A handlebody picture of $L_S$ corresponding to the schematic in Figure \ref{fig:t1}. This is a handlebody diagram for the double of the $4$-manifold $\Sigma_S^{(2)}$ (i.e. the handlebody $\Sigma_S$ without its 3- and 4-handles) together with various embedded 2-spheres.}
\label{fig:t2}
\end{figure}
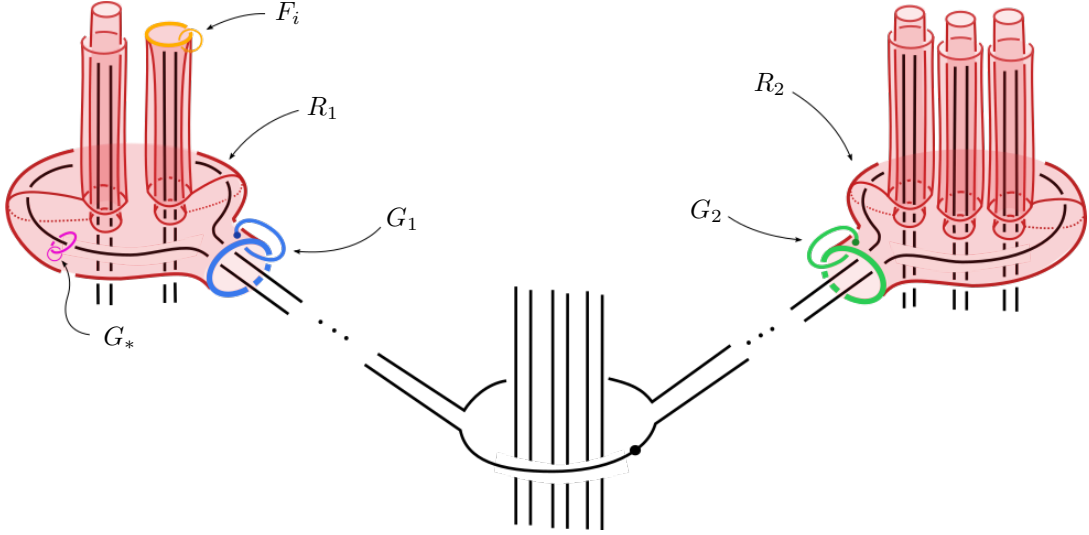

We begin by finding an embedded pre-Whitney disk $\Delta$ (as in Definition \ref{def:wm1}) pairing the cancelling intersection points between $R_1$ and $G_2$, as shown in Figure \ref{fig:t3}. The framing of the Whitney circle $\partial \Delta$ induced by $\Delta$ differs from the Whitney framing by $+1$, since $\Delta$ can be drawn as in Figure \ref{fig:t3}, passing exactly once over the $+1$-framed Gluck $2$-handle. Note that by virtue of passing over this handle, the pre-Whitney disk $\Delta$ also intersects the ascending $2$-sphere $G_*$ transversely in a single point. 

By performing a ``boundary twist" along the Whitney arc $\partial \Delta \pitchfork G_2$ (as in \cite[\S 1.3-1.4]{fqbook}), we produce an honest Whitney disk $W$. As a result of this boundary twist, this new disk now has a single intersection with the sphere $G_2$. 

\begin{figure}[ht]
\includegraphics[height=220pt]{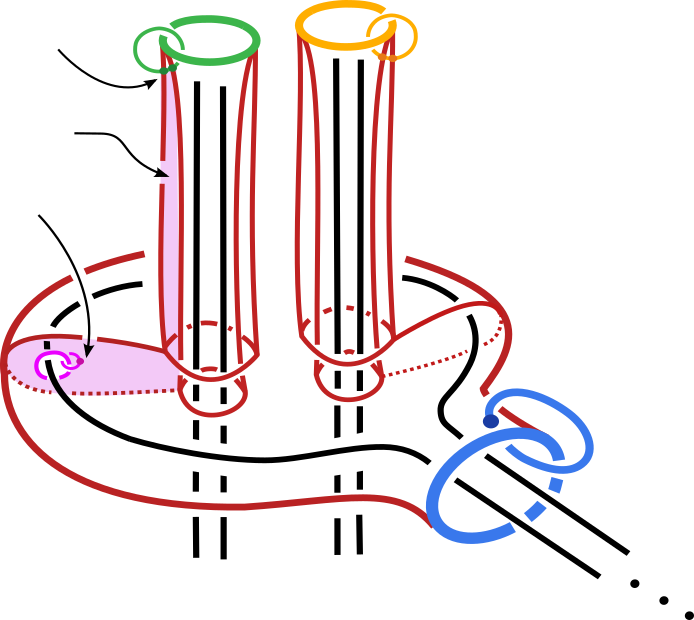}
\put(-277,203){\small in $R_1 \pitchfork G_2$}
\put(-304,215){\small Cancelling double points}
\put(-232,169){\small $\Delta$}
\put(-251,147){\small $\Delta \pitchfork G^*$}
\caption{The pre-Whitney disk $\Delta$ (shaded in pink) pairs cancelling intersection points in $R_1 \pitchfork G_2$. Note that $\Delta$ is not an honest Whitney disk, since the Whitney circle on the boundary of $\Delta$ has induced framing of $+1$ relative to the Whitney framing. Furthermore, the disk $\Delta$ intersects $G^*$ in one point, since $\Delta$ passes once over the Gluck $2$-handle.}
\label{fig:t3}
\end{figure}
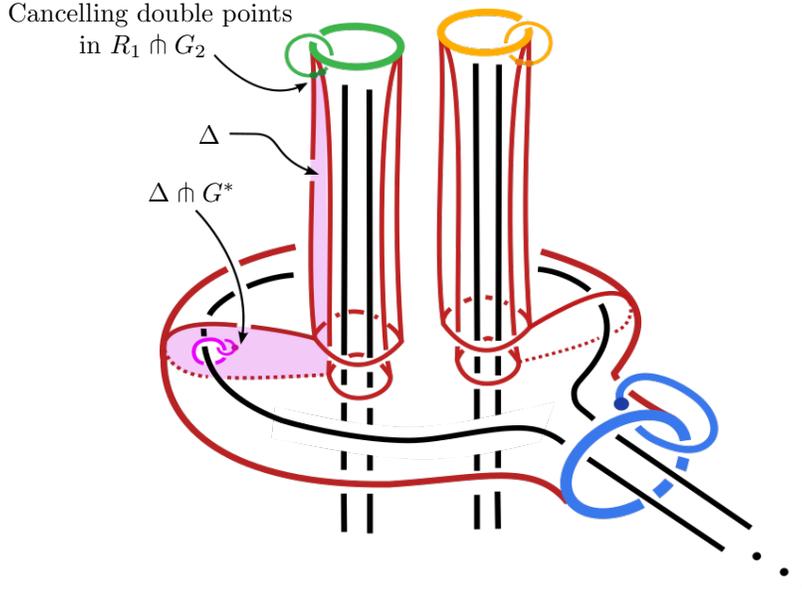

We now give a sequence of handle slides that dualize $\{R_1,R_2\}$ and $\{G_1, G_2\}$. Each step is illustrated by the corresponding schematic in Figures \ref{fig:pf1} \ref{fig:pf2}, and \ref{fig:pf3} (the colors of the handles in the previous figures have been preserved). Since we will be sliding $2$-handles, we must preserve the AC $(\ell+1)$-triviality of the induced presentation.

Let $x_1, \dots, x_{\ell +1}$ be the generators corresponding to each $1$-handle of $W_S = \Sigma_S^\circ\times I$. In addition, let $\phi_1, \dots, \phi_\ell$ be the relations induced by the $2$-handles with the ascending spheres $F_1, \dots, F_\ell$ and $r_1, r_2, r_*$ be the relations induced by the $2$-handles with the ascending spheres $G_1$, $G_2$, and $G_*$. By Lemma \ref{lem:budac}, recall that the presentation $\mathcal P$ of $\pi_1(W_S) = 1$ induced by these generators and relations is AC $(\ell +1)$-trivial (as in Definition \ref{def:ac}), since the presentation $\langle x_1, \dots, x_{\ell + 1} \mid r_*, \phi_1, \dots, \phi_\ell \rangle$ is AC-trivial.

\vspace{2mm}

\textbf{Step 1.} Initially, the intersection pattern appears schematically as in Diagram (1) from Figure \ref{fig:pf1}. Both $R_1$ and $R_2$ have intersections with $G_1$, $G_2$, and $G_*$, as well as $F_1,\dots,F_\ell$. The intersections between $R_1$ and $G_2$ are currently paired by the Whitney disk $W$, which intersects both $G_*$ and $G_2$ in a single point. 

\vspace{2mm}

\textbf{Step 2.} Slide the Gluck $2$-handle $h_*$ over $h_2$. This has the geometric effect of tubing $G_2$ over $G_*$. The result of doing so is illustrated in Diagram (2) of Figure \ref{fig:pf1}; note that $W$ now only has a single intersection with $G_*$. 

Now, we must verify that the presentation of $\pi_1(W_S)$ induced by the $1$- and $2$-handles remains AC $(\ell +1)$-trivial. Re-label the $1$-handles if necessary so that \emph{before} the slide, $r_* = x_1$ and $r_2 = x_i \omega x_j^{-1} \omega^{-1}$, where $i,j \in \{1, \dots, \ell+1\}$ and $\omega \in \langle x_1, \dots, x_{\ell+1} \rangle$. \emph{After} the handleslide, the relation $r_*$ is replaced with $r_*r_2 = x_1 x_i \omega x_j^{-1} \omega^{-1}$. However, note that $x_i= x_1^{-1} \rho$ for some product $\rho$ of the relations $\phi_1, \dots, \phi_\ell$. Indeed, there is a sequence of fission $2$-handles running between any pair of $1$-handles in $W_S$.  Therefore, we may write $r_*r_2 =\rho \omega x_j^{-1} \omega^{-1}$. It now follows that the presentation $\langle x_1, \dots, x_{\ell +1} \mid r_*r_2, \phi_1, \dots, \phi_\ell \rangle$ is AC trivial, since there is a sequence of AC moves taking it to the presentation $\langle x_1, \dots, x_{\ell +1} \mid \omega x_j^{-1} \omega^{-1}, \phi_1, \dots, \phi_\ell \rangle$, which is AC trivial by the proof of Lemma \ref{lem:budac}.

\begin{figure}[ht]
\includegraphics[height=320pt]{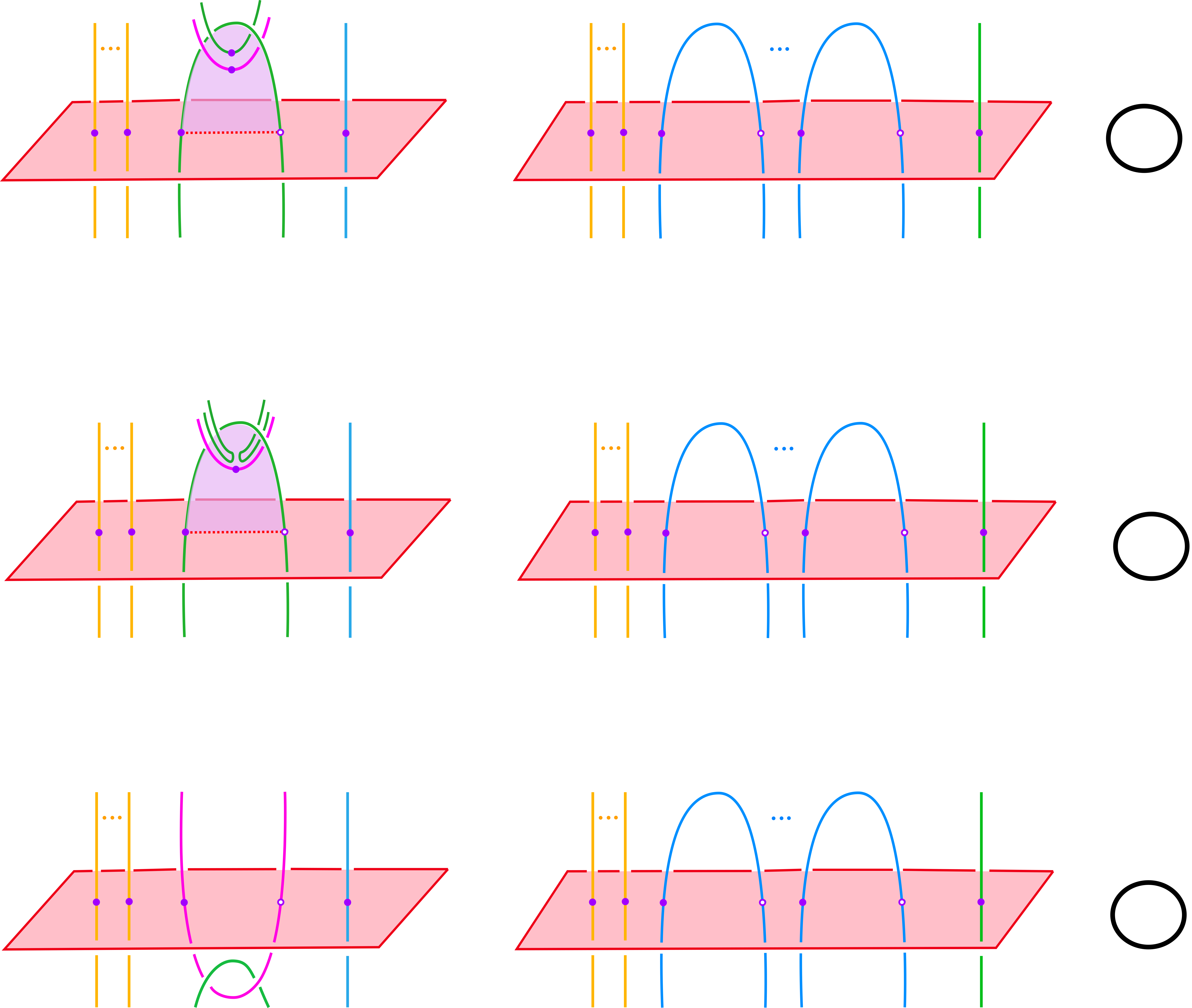}
\put(-18,273)
{\bf \large 1}
\put(-16,143)
{\bf \large 2}
\put(-17,26)
{\bf \large 3}
\caption{Steps $1, 2,$ and $3$ from the proof of Theorem \ref{thm:main}. Three-dimensional slices of the descending spheres $R_1$ and $R_2$ of the $3$-handles are illustrated in red, while the descending spheres $G_1$ and $G_2$  are colored blue and green, respectively. The ascending spheres $G_*$ and $F_1, \dots F_{\ell}$ are shown in pink and yellow. The Whitney disk $W$ (shaded in pink) is also illustrated in the first two steps.}
\label{fig:pf1}
\end{figure}

\vspace{2mm}

\textbf{Step 3.} Perform the Whitney move along the disk $W$. This has the effect of trading the two intersections of $R_1\cap G_2$ for two new intersections of $R_1\cap G_*$. Diagram (3) in Figure \ref{fig:pf1} illustrates the new schematic of intersections. This move can be  realized by an isotopy of the attaching spheres of the $3$-handles, and so does not effect the presentation of $\pi_1(W_S)$ induced by the $1$ and $2$-handles. 

\vspace{2mm}

\textbf{Steps 4 and 5.} Next, perform $3$-handle slides (which do not affect the homotopy classes of the 2-handle attaching maps) of $R_2$ over $R_1$ to remove all intersections of $R_2\cap G_1$. Diagram (4) of Figure \ref{fig:pf2} illustrates the effect on the attaching spheres, namely, tubing $R_2$ into multiple parallel copies of $R_1$. The new intersection pattern is illustrated in Diagram (5) of Figure \ref{fig:pf2}. In particular, $R_2$ now has many new intersections with $G_*$ and $F_1,\dots,F_\ell$, but is still geometrically dual to $G_2$. 

\begin{figure}[ht]
\includegraphics[height=260pt]{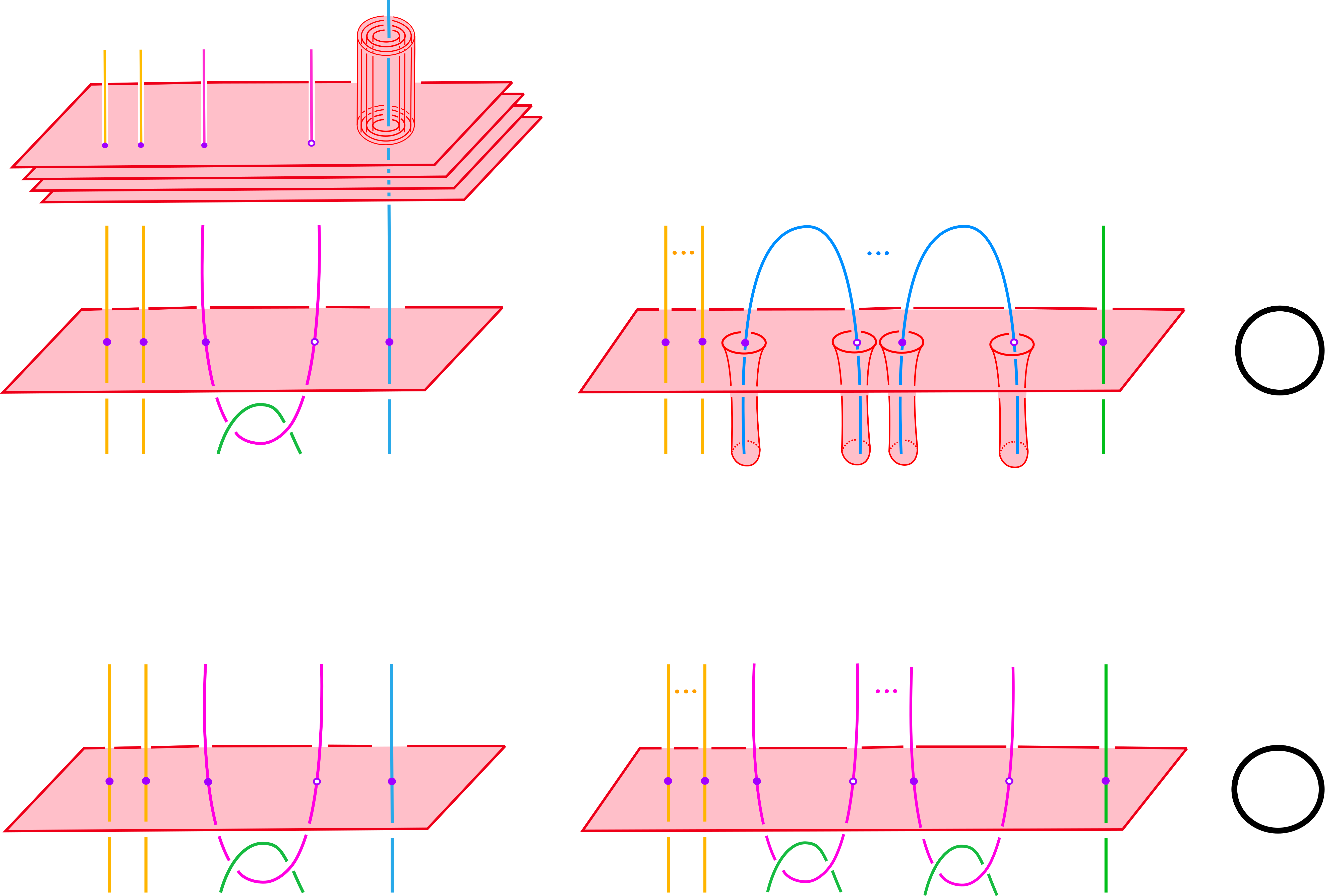}
\put(-18,155)
{\bf \large 4}
\put(-18,27)
{\bf \large 5}
\caption{Steps $4$ and $5$ from the proof of Theorem \ref{thm:main}.}
\label{fig:pf2}
\end{figure}

\vspace{2mm}

\textbf{Step 6.} As in Diagram (6) of Figure \ref{fig:pf3}, tube $G_*$ and $F_1,\dots,F_\ell$ over $G_1$ so that $R_1$ only intersects $G_1$. Note that since $R_2\cap G_1=\emptyset$, this does \emph{not} change any intersections of $R_2$. This corresponds to sliding $h_1$ over the other corresponding 2-handles; while this may change the relation corresponding to $h_1$, it does not affect the homotopy classes of the attaching regions for the remaining handles. 

\vspace{2mm}

\textbf{Step 7.} As in Diagram (7) of Figure \ref{fig:pf3}, tube $G_*$ and $F_1,\dots,F_\ell$ over $G_2$ so that $R_2$ has no extraneous intersections. As in Step 6, this corresponds to sliding $h_2$ over various 2-handles, and so does not affect the homotopy classes of attaching regions of the remaining handles.

\begin{figure}[ht]
\includegraphics[height=300pt]{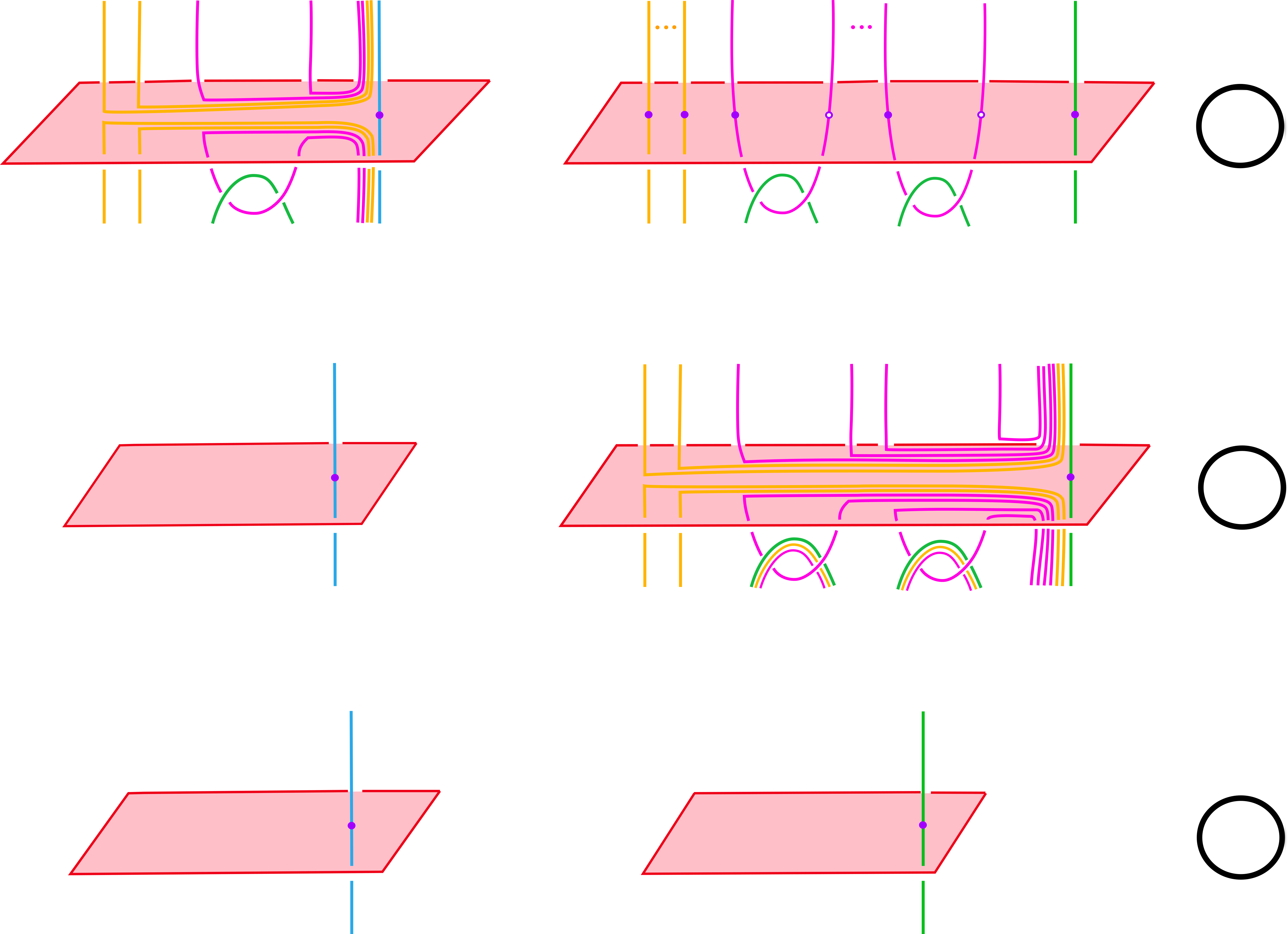}
\put(-19,256)
{\bf \large 6}
\put(-19,27)
{\bf \large 8}
\put(-18,139)
{\bf \large 7}
\caption{Steps $6, 7,$ and $8$ from the proof of Theorem \ref{thm:main}.}
\label{fig:pf3}
\end{figure}

\vspace{2mm}

\textbf{Step 8.} At this point, $\mathcal{R}$ and $\mathcal{G}$ are now dual (Figure \ref{fig:pf3} (8)), and can be geometrically cancelled. 

\vspace{2mm}

The remaining 1- and 2-handles induce an Andrews-Curtis trivial presentation of $\pi_1(W_S)$, and so by Lemma \ref{lem:ACproduct} we conclude that $\Sigma_S^\circ\times I\cong B^5$. \qed

\section{Applications and questions}\label{sec:examples}

We end the paper with some applications of Theorem \ref{thm:main}, along with some questions. 

There are ribbon presentations for each ribbon knot with less than 11 crossings given in \cite[F.5]{Kaw96book}; many of these specific disks have undisking number equal to one. By Lemma \ref{lem:kawauchidisks} we obtain the following corollary.

\begin{corollary}\label{cor:kawauchidisks} 
Suppose that a 2-sphere $S$ has one ribbon hemisphere given by a ribbon disk for one of the following knots in \cite[F.5]{Kaw96book}. Then $\Sigma^\circ_S\times I\cong B^5$.
\[\{6_1,8_8,8_9,8_{20}, 9_{27}, 9_{41}, 9_{46}, 10_{48},10_{75}, 10_{87},10_{129},10_{137}, 10_{140},10_{153},10_{155}\}\]
\end{corollary}

\begin{remark}\label{rem:gompfex} 
    Some of these knots have symmetries which produce different ribbon disks. In particular, Cappell and Shaneson's famous homotopy 4-sphere (now known to be $S^4$) is the Gluck twist on a sphere given as the union of two ribbon disks for $8_9$ which are related by an involution (see e.g. \cite{Kirby}). Theorem \ref{thm:main} offers an alternate proof that the double of this homotopy sphere is standard and may apply to homotopy spheres constructed similarly using the list above. 
    
    A simple example is the 2-twist spun trefoil, which serves as the guiding example of this paper. It can be presented as the union of two ribbon disks for $6_1$, both of which have undisking number one (see the footnote in Figure \ref{fig:foxbud}).  
\end{remark}

Theorem \ref{thm:main} also applies to some deformation spun knots. By Lemma \ref{lem:undiskingnumtwistknot}, we obtain the following. 

\begin{corollary}\label{cor:twistcor}
Suppose that $K$ is a twist knot, and let $D_K\subset B^4$ be the ribbon disk for $K\#-K$ obtained by spinning $K$ through the 4-ball. If a 2-sphere $S$ has one ribbon hemisphere given by $D_K$, then $\Sigma_S^\circ\times I\cong B^5$.
\end{corollary}

Twist knots have unknotting number one, so twist-roll spins of these knots are already covered by the main theorem in \cite{NaySch20}. However, because one of the two ribbon hemispheres may be arbitrary, this corollary is more general (at the cost of a weaker conclusion). 

Since all Gluck twists are conjectured to be standard, one might expect that $\Sigma_{S\#-S}$ is always standard. As a starting point, it may be fruitful to try to push the techniques of \S \ref{sec:mainresults} further. 

\begin{question}
    For what other classes of ribbon disks does an analogue of Theorem \ref{thm:main} hold? In particular, what if one ribbon hemisphere has fusion number equal to one (but arbitrary undisking number)?
\end{question}

Another important consequence of this theorem is that at least some (possibly non-standard) Gluck twists embed in the 4-sphere. Recall from the introduction that a homotopy 4-ball is called \bit{Schoenflies} if it embeds in the 4-sphere. 

\begin{theorem}\label{thm:schoenfliesball}
Suppose that a 2-sphere $S\subset S^4$ admits a decomposition into two ribbon disks, one of which has undisking number equal to one. Then, $\Sigma_S$ is a Schoenflies ball. 
\end{theorem}

Thus, Theorem \ref{thm:schoenfliesball} (in combination with the previous two corollaries) gives a variety of new potential counterexamples to the 4-dimensional Schoenflies conjecture. Interestingly, by a theorem of Gompf \cite{Gom91}, embeddings of such homotopy 4-balls in $S^4$ (and consequently their complements) are unique up diffeomorphism (in fact, up to smooth isotopy) and so proving such Gluck balls are standard is equivalent to verifying the Schoenflies conjecture in this case (see \cite{Gab22}). We end with the following interesting question.

\begin{question}
    Are all Gluck twists diffeomorphic to Schoenflies balls? In other words, does $\Sigma_S^\circ$ always smoothly embed in $S^4$?
\end{question}

\bibliographystyle{alpha}
\bibliography{references}

\end{document}